\documentclass[11pt,reqno,a4paper]{extarticle} 

\usepackage[osf]{newtxtext}
\usepackage[left=3cm,right=3cm,top=2.2cm,bottom=2.2cm]{geometry}
\RequirePackage{amsthm,amsmath,amsfonts,amssymb}
\RequirePackage[numbers]{natbib}
\usepackage{mathrsfs}
\usepackage{stmaryrd}
\SetSymbolFont{stmry}{bold}{U}{stmry}{m}{n}
\usepackage{mathtools} 
\usepackage{enumerate,enumitem}
\usepackage[pdfencoding=auto,hidelinks]{hyperref}
\usepackage{esint} 
\usepackage{empheq}
\usepackage{graphicx}
\usepackage{titlefoot} 
\usepackage{authblk} 

\theoremstyle{plain}
\newtheorem{theorem}{Theorem}[section]
\newtheorem{proposition}[theorem]{Proposition}
\newtheorem{lemma}[theorem]{Lemma}
\newtheorem{corollary}[theorem]{Corollary}
\newtheorem{assumption}[theorem]{Assumption}

\theoremstyle{definition}
\newtheorem{definition}[theorem]{Definition}
\newtheorem{notation}[theorem]{Notation}
\newtheorem{remark}[theorem]{Remark}
\newtheorem{example}[theorem]{Example}

\newcommand{\N}{\ensuremath{\mathbb N_0}} 
\newcommand{\R}{\ensuremath{\mathbb R}} 

\DeclareMathOperator{\divergence}{div}
\renewcommand{\div}{\ensuremath{\divergence}}
\DeclareMathOperator*{\esssup}{ess\,sup}
\DeclareMathOperator*{\essinf}{ess\,inf}
\newcommand{\dd}{\ensuremath{\mathrm d}}
\newcommand{\T}{\ensuremath{{\mathbb T}}}
\newcommand{\X}{\ensuremath{{\mathbf X}}}
\newcommand{\Y}{\ensuremath{{\mathbf Y}}}
\newcommand{\Z}{\ensuremath{{\mathbf Z}}}
\newcommand{\Q}{\ensuremath{{\mathbf Q}}}
\newcommand{\LL}{\ensuremath{{\mathbb L}}}
\newcommand{\XX}{\ensuremath{{\widehat X}}}
\newcommand{\XXX}{\ensuremath{{\widehat X}^{\odot2}}}
\newcommand{\B}{\ensuremath{{\mathbf B}}}
\newcommand{\CC}{\ensuremath{{\mathcal C}^{1+}}}
\newcommand{\sgn}{\ensuremath{{\mathrm{sgn}\,}}}
\makeatletter
\newcommand*\bigcdot{\mathpalette\bigcdot@{.5}}
\newcommand*\bigcdot@[2]{\mathbin{\vcenter{\hbox{\scalebox{#2}{$\m@th#1\bullet$}}}}}
\makeatother
\newcommand{\bb}[2]{\ensuremath{(#1\bigcdot#2)}}
\newcommand{\bbb}[2]{\ensuremath{(#1\odot#2)}}

\usepackage{xcolor}

\begin{document}

\title{Quasilinear rough partial differential equations with transport noise}

\date{}

\author{Antoine Hocquet}

\affil{\small Technische Universit\"at Berlin, Institut f\"ur Mathematik,  Stra\ss e des 17. Juni 136, 10623 Berlin, Germany}

\maketitle

\unmarkedfntext{\textit{Mathematics Subject Classification (2020) ---} Primary 60L50, 60L20; Secondary 35K59, 35A15, 60H15}

\unmarkedfntext{\textit{Keywords and phrases ---} rough partial differential equations, rough paths,  quasilinear equations, variational methods, stochastic partial differential equations, renormalized solutions}

\unmarkedfntext{\textit{Mail}: antoine.hocquet@tu-berlin.de }

\begin{abstract}
We investigate the Cauchy problem for a quasilinear equation with transport rough input of the form
$\dd u-\partial _i(a^{ij}(u)\partial _j u)\dd t =\dd \X_t^i(x)\partial_i u_t,$ $u_0\in L^2$ on the torus $\mathbb T^d$,
where $\X$ is two-step enhancement of a family of coefficients $(X^i_t(x))_{i=1,\dots d}$, akin to a geometric rough path with H\"older regularity $\alpha>1/3.$
Using energy estimates, we provide sufficient conditions that guarantee existence in any dimension, and uniqueness in the case when $X$ is divergence-free. We then focus on the one-dimensional scenario, with slightly more regular coefficients. Improving the a priori estimates of the first results, we prove existence of a class of solutions whose spatial derivatives satisfy a Ladyzhenskaya-Prodi-Serrin type condition.
Uniqueness is shown in the same class, by obtaining an $L^\infty(L^1)$ estimate on the difference of two solutions. The latter is obtained by establishing a link with a certain backward dual equation combined with a (rough) iteration lemma \`a la Moser.
\end{abstract}

\tableofcontents

\section{Introduction}

We consider the Cauchy problem for a quasilinear non-degenerate parabolic rough partial differential equation of the form
\begin{equation}
\left\{
\begin{aligned}
\label{problem}
&\dd u_t - \partial _i(a^{ij}(t,x,u)\partial _j u_t)\dd t = \dd \X^i_t(x)\partial _i u_t,\quad [0,T]\times\T^d,
\\
&u_0=u^0\in L^2(\T^d)
\end{aligned}\right.
\end{equation}
(with implicit summation over repeated indices),
where $\dd \X$ denotes time-integration with respect to a geometric enhancement
of an $\alpha $-H\"older ``rough sheet'' $t\mapsto X_t(\cdot)$, consisting in a pair
\begin{equation}\label{latter_pair}
\X=(X^i_t(x),\LL_t^i(x))_{i=1\dots d},\quad t\in[0,T],\enskip x\in \mathbb T^d\,,
\end{equation}
where H\"older regularity is to be understood with respect to the variable $t$ (for almost every $x\in \mathbb T^d$) and throughout the paper, 
\begin{equation}\label{regularity}
\alpha >1/3\,.
\end{equation}
A possible motivation for the study of \eqref{problem} lies in HJB equations associated with stochastic optimal control problems, specifically in presence of mean-field type interactions. The rough term $X$ accounts for a fixed trajectory of an external, independent source of noise that can be considered as an observation. See for instance \cite{buckdahn2017,bonnans2018} and the references therein.

In the context of \eqref{problem}, the pair \eqref{latter_pair} can be built on Lyons' insight that the rough input $X$ must be accompanied with additional information, such as iterated integrals of $X$ with itself \cite{lyons1998differential,lyons2002system}. 
For a path living in a functional space $V$ over the $x$-variable (as is the case of $X$ here with $V=W^{3,\infty}(\mathbb T^d)$), iterated integrals are usually understood as elements of the tensor product $V\otimes V$, and hence can thought of as functions of \textit{two} spatial variables $x,y\in \mathbb T^d.$
Here, the enhancement we rely on is much simpler, since the second component $\LL^i$ should be understood as a prescribed value for the generally ill-defined iterated integral
\begin{equation}\label{L_intro}
\LL^i_{st}(x)=\int_s^t \dd X_r^\mu (x) (\partial _\mu X^i_r(x)-\partial _\mu X^i_s(x))\,.
\end{equation}
This simplification occurs as a result of the presence of symmetries in \eqref{problem}, which are themselves consequences of Schwarz' theorem that $\partial_{ij}=\partial_{ji}$, together with the geometricity of $X$. For related investigations, we refer to \cite[Section 2]{hocquet2018ito}.
Note that in Section \ref{sec:parabolic} below we shall also employ the more suggestive notation $\bb{X^\mu}{\partial _\mu X}^i$ instead of $\LL^i$,
but it will be always considered as part of the data. 
As for the left hand side in \eqref{problem}, we assume the following strong parabolicity condition for $a^{ij}.$
\begin{assumption}
	\label{ass:A}
	$a\colon[0,T]\times \mathbb T^d\times\R\to \mathscr L(\R^d,\R^d)$
	is symmetric, measurable, $C^1$ with respect to the third variable, and such that
	\begin{equation}
	\label{coercivity}
	\lambda |\xi |^2\leq a^{ij}(t,x,z)\xi ^i\xi ^j\leq \lambda ^{-1}|\xi |^2,
	\end{equation} 
	for some $\lambda>0,$ independent of $\xi \in \R^d$ and $(t,x,z)\in [0,T]\times\mathbb T^d\times\R.$ 
\end{assumption}

Our main results deal with existence and uniqueness of solutions to the problem \eqref{problem}, under the minimal requirement that the solution belongs to the energy space $L^\infty(L^2)\cap L^2(W^{1,2})$ and under various conditions on the rough forcing term.
It is certainly possible to generalize our results to a broader class of quasilinear equations for instance by adding a flux term of the form $\div(F(u))\dd t$, where
$F\equiv(F^1,\dots , F^d)\colon \R\to \R^d$
is $C^1$ and bounded.
Since this would complicate the algebra without requiring any striking argument, we however restrain from doing so and focus on the ansatz \eqref{problem}.
Moreover, working with \textit{branched} rough paths, the geometricity assumption as well as the fact that $\alpha >1/3$ could certainly be weakened. See for instance \cite{bellingeri2020transport} for an attempt in this direction, in the context of a (pure) transport equation.

\paragraph*{Comparison with existing approaches (and further motivations)}
As is widely acknowledged, Lyons' theory of rough paths works quite well in finite-dimensions, where rough paths enhancements (i.e.\ \textit{signatures}) of a given path of low-regularity are used as a tool to solve controlled differential equations in a robust way. 
For inputs which have a stochastic nature (Gaussian process or Markov, provided a meaningful stochastic area can be built), it provides an apparatus to construct stochastic flows, prove large deviations estimates or support theorems (among other things).
Because the concept of signature does not seem to generalize well to the case of several variables, the application of Lyons' theory to rough \textit{partial} differential equations is on a case-by-case basis.
In particular, a complete understanding of an appropriate analytic framework (similar to \cite{krylov1999analytic} in the stochastic case) is still missing.
 The present contribution is an attempt to set up such a framework by following the lines of the ``variational approach'' developed in \cite{bailleul2017unbounded,deya2016priori,hocquet2017energy,hofmanova2018navier,hocquet2018ito,hofmanova2020vorticity}.

The problem \eqref{problem} was independently investigated by Friz, Caruana and Oberhauser \cite{caruana2011rough,friz2014rough}, under different assumptions on the coefficients (see also \cite{seeger2018}). 
These authors considered a non-linear equation of the form
\begin{equation}
\label{quasilin_viscous}
\dd u - F(t,x,u,Du,D^2u) \dd t = H_i(t,x,u,D u)\dd \Z^i\,\quad\text{on}\enskip[0,T]\times\R^d
\end{equation}
where the left hand side $F$ is continuous and degenerate-elliptic, whereas $H(t,x,r,p)$ is continuous and affine linear with respect to $r\in\R$ and $p\in\R^d$.
Their analysis relies on a viscosity approach in the spirit of the pionneer work of Lions and Souganidis \cite{lions1998fully}, combined with flow transformation techniques. Variants of \eqref{quasilin_viscous} include the case of a non-linear rough flux, as investigated for instance by Lions, Perthame and Souganidis \cite{lions2013scalar}. In this regard, we should mention the works of Gess and Souganidis \cite{gess2015scalar,gess2017parabolic,gess17cpam} which make use of test functions and Sobolev spaces (similar to our setting), based on a formalism of pathwise entropy/kinetic solutions. 
In a later contribution of Fehrman and Gess \cite{fehrman2019well} an analogous model is considered without relying on a flow-transformation, but considering instead a certain family of locally fluctuating test functions. 

It should be observed that the technique of \textit{renormalized solution} that we use below in order to obtain $L^1$ estimates (see Theorem \ref{thm:renorm}) carries some similarities with the type of arguments encountered in the theory of entropy/kinetic solutions.
In fact, as is classical (see \cite{perthame2002kinetic} and the references therein), renormalized solutions can be seen as a particular case of the latter. On the other hand, it is possible to show that the solutions of finite energy considered in theorems \ref{thm:existence}, \ref{thm:existence_2} below are alternatively characterized by the renormalization property \eqref{chain_rule_thm}. 
Though some of our main results might therefore overlap with that of the previous authors, we point out that our assumptions and conclusions are of a different nature. 
For instance, the formalism we rely on to enhance the rough coefficient $X_t(x)$ to the input $\X$ in \eqref{problem} is not standard, in that \eqref{L_intro} differs from what is encountered in usual rough paths theory \cite{friz2014course}. It is shaped by the formalism of ``unbounded rough driver'' introduced in \cite{bailleul2017unbounded,deya2016priori} and subsequently analysed in \cite{hocquet2017energy,hocquet2018ito,hofmanova2018navier}. 
The main reason for working with such objects is in fact related to our aim to exploit certain symmetries due to the transport nature of the noise in \eqref{problem}. In the present context, it appears that iterated integrals of the form $\int_s^t X_{sr}\otimes\dd X_r$ (as suggested by the usual theory) are not exactly the right quantities to look at. See Example \ref{example:bracket} for a related discussion. 
This aspect of our treatment is further emphasized in the ``weak geometric relations'' (Lemma \ref{lem:weak_geo} below), which play a central role in the proof that uniqueness holds when $X^i_t(x)$ is divergence-free (see Theorem \ref{thm:uniqueness}). 
In a stochastic context, this echoes the works of Ciotir and T\"olle \cite{ciotir2016nonlinear,tolle2020stochastic}, where commutator relations in terms of Killing vector fields are derived, leading to similar cancellations.

If there is a merit in our approach, it is its relative simplicity and its compliance with classical PDE techniques (such as energy estimates, maximum principles, duality arguments etc.; see \cite{ladyzhenskaya1968linear}).
Moreover, our notion of solution appears to blend well with Gubinelli's concept of \textit{controlled paths} \cite{gubinelli2004controlling}, defined here in terms of suitable Sobolev spaces (see Definition \ref{def:controlled}). 
For all these reasons, the proofs below are essentially independent from the existing literature on non-linear rough PDEs.

\paragraph*{Stochastic bibliography}
Stochastic analogues of \eqref{problem} constitute a basic motivation for introducing a rough paths formulation, where the noise $\dot X$ is often used to account for physical uncertainties. In this direction, we point out the pioneering work of Denis and Stoica \cite{denis2004general}, allowing for a transport noise close to the rough input in \eqref{problem}. More recently, Munteanu and R\"ockner \cite{munteanu2016total} have studied the so-called \textit{total variation flow} perturbed by a similar gradient noise (see also \cite{ciotir2016nonlinear,barbu2018stochastic,dareiotis2020nonlinear}).
During the past decade, significant progress has been made in the context of rough/stochastic scalar conservation laws (see, e.g., \cite{debussche2010scalar,lions2013scalar,lions2014scalar,friz2014stochastic,hofmanova2015scalar,debussche2016degenerate,gess2018well}, i.e.\ equations of the form
\begin{equation}
\label{conserv}
\dd u - \partial_i(a^{ij}(x,u)\partial_j u)\dd t = \partial _i(F^i(u,x))\circ \dd X+ \Phi(u)\circ \dd W
\end{equation}
where here $W$ denotes some Wiener process with trace-class covariance, $X$ is another source of noise, typically another Wiener process or a rough path (both finite-dimensional in the previous references) and the symbol $\circ$ stands either for the Stratonovitch product between semimartingales, or for the usual product between a controlled path with a (geometric) rough path.
The flux term $F$ is generally assumed to be of class $C^2$ in space, with some polynomial growth on $F''.$
In constrast with the present contribution however, the matrix $a^{ij}(u)$ is allowed to degenerate. This pushes the previous authors to look at the corresponding kinetic formulation, in which case a transport noise of the same form as \eqref{problem} naturally appears. Note that in  the previous references (with the exception of \cite{hofmanova2015scalar}), the authors deal with a version of \eqref{conserv} where a noise term is present either in the flux or in the multipicative form $\Phi(u)\circ dW$.
In the case of a non-degenerate matrix $a^{ij}$,
assuming $X_t=t$, Hofmanov\'a and Zhang \cite{hofmanova2017quasilinear} were able to provide a rather simple proof of existence and uniqueness for \eqref{conserv}. Using entirely different techniques, non-degenerate parabolic sytems with multiplicative noise as above (but without transport term) were recently investigated by Kuehn and Neamtu \cite{kuehn2018pathwise}.

\paragraph*{Singular quasilinear PDEs}
Based on the corresponding semilinear theories 
\cite{hairer2013kpz,hairer2014theory,gubinelli2015paracontrolled},
another major direction of research in the past years 
was the analysis of quasilinear \textit{singular} SPDEs, i.e.\ equations of the form
\begin{equation}
\label{quasilin_singular}
\partial _tu - A(u)u  = F_0(u,\nabla u) + F_1(u)\dot X \,\quad \text{on}\enskip[0,T]\times\T^d
\end{equation}
($F_0$ and $F_1$ are smooth functions),
where this time $\dot X$ lies in the parabolic H\"older space $\mathcal C^\alpha([0,T]\times\mathbb T^d)$ for some negative $\alpha$, and hence is irregular in both time and space. See the works  \cite{otto2016quasilinear,otto2019divergence,furlan2016paracontrolled,bailleul2019quasilinear,gerencser2019solution}.

The two model equations \eqref{problem} and \eqref{quasilin_singular} are, in fact, thoroughly different. First, \eqref{problem} is not really singular, in the sense that no renormalization step is required in the analysis. This fact allows us to take \textit{products} of \eqref{problem} with equations driven by similar rough inputs (see Appendix \ref{app:product}), and thereby to obtain suitable $L^p$-estimates. 
Regarding \eqref{quasilin_singular} however, deriving an equation for non-linear fonctions/products of such solutions remains an open problem (see nevertheless \cite{zambotti2006ito,bellingeri2018ito} for attempts in that direction).
Second, \eqref{problem} and \eqref{quasilin_singular} belong to a different class of PDEs, as can be seen as follows. While in principle, the nonlinearity $F_1$ in \eqref{quasilin_singular} could depend on the gradient of solutions, this would in turn impose important limitations on the time-regularity of $X$. Loosely speaking, such limitations guarantee that \eqref{quasilin_singular} behaves at ``small scales'' as a non-degenerate parabolic equation (see \cite[Chap.~8]{hairer2014theory} for a precise meaning).
This is definitely not the case of \eqref{problem}, where interesting examples show that the latter property can be violated (e.g.\ when $X$ is a fractional Brownian motion with Hurst parameter $H\in(\frac13,\frac12)$, see \cite{hocquet2018ito} for a related discussion). We should however mention the works \cite{delarue2016,cannizzaro2018} where a transport rough term with a \textit{spatial} regularity similar to \eqref{regularity} is considered (informally speaking, this is permitted by the fact that in a parabolic context, spatial regularity ``counts half as much'' as time regularity).

\paragraph{Organization of the paper}
The forthcoming Section \ref{subsec:settings} is devoted to definitions and settings, while in Section \ref{subsec:results} we give our main existence and uniqueness results for \eqref{problem}.
In Section \ref{sec:preliminaries} we give some preliminary results that are useful for the sequel such as the renormalization property or the Rough Gronwall Lemma.
Existence, and the proof of Theorem \ref{thm:existence} will be given in Section \ref{sec:existence}.
The first uniqueness theorem, Theorem \ref{thm:uniqueness} will be proved in Section \ref{sec:uniqueness}.
In Section \ref{sec:parabolic}, solvability and $L^\infty$  bounds for a class of parabolic equations are investigated, the formulation of which are necessary for the understanding of further sections. Similarly, Section \ref{sec:transport} deals with a rough transport equation with additive rough input, whose study is motivated by the proof of the second uniqueness theorem, Theorem \ref{thm:uniqueness}. The proof of the latter will then be addressed in Section \ref{sec:uniqueness_2}, while the second existence result, Theorem \ref{thm:existence_2}, is shown in Section \ref{sec:higher}. Finally, Appendix \ref{app:sewing} recalls the statement of the Sewing Lemma, while Appendix \ref{app:product} is devoted to the product formula between solutions of two distinct rough PDEs.

\subsection{Notation}
Throughout the paper we consider a finite, fixed time horizon $T>0.$
By $\mathbb N,$ we denote the set of natural integers $1,2,\dots,$ and we let $\N:=\mathbb N\cup \{0\}.$ Real numbers are denoted by $\R$ and we also adopt the notation $\R_+:=[0,\infty).$\smallskip

We will work with the usual Lebesgue and Sobolev spaces in the space-like variable: $L^p(\mathbb T^d)$, $W^{k,p}(\mathbb T^d),$ for $k\in\N,$ and $p\in[1,\infty],$ whose corresponding norms will be denoted by ${|\cdot |_{L^p(\mathbb T^d)}},$ ${|\cdot |_{W^{k,p}(\mathbb T^d)}}.$
The shorthand notations $L^p$ and $W^{k,p}$ will be sometimes used for simplicity.
For functions $f$ also depending on the time-like variable, we use the notation 
\[
\|f\|_{L^r(s,t;L^q)}
:=\left(\int_{s}^t\left(\int_ {\T^d} |f_\tau (x)|^q\dd x\right)^{r/q}\dd \tau \right)^{1/r},\quad \text{for each}\enskip 0\leq s\leq t\leq T,
\]
and we will also write $\|f\|_{L^r(L^q)}$ instead of $\|f\|_{L^r(0,T;L^q)}.$
By $C(0,T;E),$ we denote the space of continuous functions from $[0,T]$ taking values in a Banach space $E.$ It is endowed with the usual supremum norm. 

Given another Banach space $F,$ we will denote by $\mathscr L(E,F)$ the space of linear, continuous maps from $E$ to $F,$ endowed with the operator norm. For $f$ in $E^*:=\mathscr L(E,\R),$ we will denote the dual pairing by
$\langle f,g \rangle$ (i.e.\ the evaluation of $f$ at $g\in E$).

\paragraph*{Rough paths.}
We now introduce some notation related to the theory of controlled rough paths \cite{lyons1998differential,gubinelli2004controlling}.
We will denote by $\Delta ,\Delta _2$ the simplices
\begin{equation}\label{nota:simplexes}
\begin{aligned}
&\Delta:=\{(s,t)\in [0,T]^2\,,\,s\leq t\}\,,
\\
&\Delta _2:= \{(s,\theta ,t)\in [0,T]^3\,,\,s\leq \theta \leq t\}\,.
\end{aligned}
\end{equation}
If $E$ is a vector space and $g\colon[0,T]\to E$ is a continuous path, we make use of the standard rough path notation
\[
g_{st}:= g_t - g_s,\quad  \text{for all}\quad (s,t)\in\Delta \,,
\]
hence blurring the difference between the value of $g$ and its increments.
For a general two-index map $h=(h_{st})_{(s,t)\in\Delta }$, we define an element $\delta h$ on the simplex $\Delta _2$ as follows:
\[
\delta h_{s\theta t}:=h_{st}-h_{s\theta }-h_{\theta t},\quad \text{for every}\enskip (s,\theta ,t)\in\Delta _2,
\]
and we recall that the kernel of $\delta $ is precisely the set of increments $g_{st}=g_t-g_s$ as before.

If $E$ is equipped with a norm $|\cdot |_{E},$ we shall denote by 
$C_ 2^{\alpha }(0,T;E)$ the set of $2$-index maps $g:\Delta \to E$ such that $g_{tt}=0$ for every $t\in [0,T]$ and 
\begin{equation}
\label{def:omega_a_2}
[g]_{\alpha ,E}:=\sup_{(s,t)\in\Delta }\frac{|g_{st}|_{E}}{(t-s)^\alpha }<\infty.
\end{equation}
We also denote by $C^\alpha (0,T;E)$ the space of paths $g\colon[0,T]\to E$, such that $(s,t)\mapsto g_{st}$ belongs to $C_2^\alpha (0,T;E)$ (this corresponds to the usual H\"older space).

\paragraph*{Controls.}
In the manuscript we will also encounter functions on the simplex $\Delta $ that behave with continuity properties similar to H\"older, but in a weaker sense. For that purpose  we need to introduce the notion of control function.
We call \textit{control} on $[0,T]$ any continuous and superadditive map $\omega \colon\Delta \to \R_+,$ that is, for all $(s,\theta ,t)\in \Delta _2$ it holds the inequality
\begin{equation}
\label{axiom:control}
\omega (s,\theta )+\omega (\theta ,t)\leq \omega (s,t)\,,
\end{equation}
implying in particular that $\omega (t,t)=0.$
Example of controls are given by $\int_s^tf_r\dd r$ with $f\geq 0$, $f\in L^1$, or by $\omega(s,t)^\alpha\varpi(s,t)^\beta$ for another such $\varpi$ and every $\alpha,\beta\geq0,$ with $\alpha+\beta\geq 1$ 
(see \cite{friz2010multidimensional} for a thorough presentation).

We will denote by 
\begin{equation}
\label{nota:Z}
\CC(0,T;E)
\end{equation} 
the space consisting of $g\in C_2^0(0,T;E)$ such that there exists a constant $\ell >0,$ a control $\omega ,$ and a real number $z>1$ so that $|g_{st}|_{E}\leq \omega (s,t)^z,$ for every $(s,t)\in\Delta $ such that $|t-s|\leq \ell .$

\subsection{Unbounded rough drivers}
\label{subsec:settings}

We assume throughout the paper that the coefficient path $t\mapsto X_t(x)$ is $\alpha $-H\"older, for Lebesgue almost every $x\in\T^d$,
with the hypothesis
\begin{equation}
\label{ass:alpha}
\alpha >1/3\,.
\end{equation} 
As is well-known in rough paths theory, properly formulating \eqref{problem} requires to assume further knowledge on $X$, such as an ``enhancement'' $\X=(X,\LL)$. 
The first component is nothing but the path itself, while here $\LL$ contains consistent additional data related to the iterated integrals.
Precisely, we make the following assumption.

\begin{assumption}[Geometric enhancement of the coefficients]
	\label{ass:geometric}
	We are given a path $X\in C^\alpha (0,T;(W^{3,\infty})^d)$ and an enhancement $\X_{st}=(X_{st},\LL_{st})$ of $X$ such that
	\[X_{st}=X_t-X_s\,,\quad (s,t)\in\Delta\,\,,\]
	while $\LL=(\LL^i)_{i=1,\dots,d}$ is subject to the Chen's-type relation
	\begin{equation}
	\label{chen_L}
	\delta \LL^i_{s\theta t}=\partial _\mu X_{s\theta }^iX_{\theta t}^\mu \,,\quad (s,\theta,t)\in\Delta_2\,.
	\end{equation} 
	
	We assume that $\X$ is \emph{geometric}, by which we mean that there exists a sequence $X(n)\in C^1(0,T;W^{3,\infty}(\T^d)),$ $n\geq 0$ such that if $\X_{st}(n)\equiv (X_{st}(n),\LL_{st}(n))$ is defined as the \emph{canonical lift}
	\[
	\left\{
	\begin{aligned}
	&X^{i}_{st}(n):= X_t^i (n)-X^i  _s(n)\,,
	\\
	&\LL^{i}_{st}(n):= \int_s^t \dd X^\mu_r(n)\partial _\mu X_{sr}^i(n)\,,
	\end{aligned}\right.\quad \text{for}\enskip i=1,\dots ,d,
	\]
	(the second term is a Bochner integral in the space $W^{2,\infty}$),
	then it holds
	\begin{align}
	\label{rho_alpha}
	\rho _\alpha (\X(n),\X)\enskip :=
	\|X(n)-X\|_{C^\alpha (0,T;W^{3,\infty})}
	+[\LL(n)-\LL]_{2\alpha;W^{2,\infty}}\to0.
	\end{align}
\end{assumption}

Attached to such enhancements $\X$ is the notion of unbounded rough driver.
The following is an affine linear generalization of the definitions given in \cite{bailleul2017unbounded,deya2016priori,hocquet2017energy,hofmanova2018navier,hocquet2018ito}, and will be needed in the sequel.
\begin{definition}
	\label{def:Q}
	Assume that we have a pair of two-index maps $\Q=(Q^1_{st},Q^2_{st})_{(s,t)\in\Delta }$ such that for each $i=1,2,$ and $k=-3+i,\dots 3,$ the map 
	\[
	Q^i_{st}:W^{k,p}\to W^{k-i,p}
	\]
	is well-defined for all $p\in[1,\infty]$, affine linear and bounded. Moreover, denoting by $(\mathrm{Lip}(E,F),$ $\nolinebreak{|\cdot |_{\mathrm{Lip}(E,F)}})$ the usual space of Lipshitz maps between normed vector spaces $E$ and $F$, we further assume that
	\begin{enumerate}
		\item [($\star$)]\label{RD1}
		for $i=1,2,$ and $k=-3+i,\dots ,3,$ 
		$Q^i$ belongs to $C^{i\alpha }_2\left(0,T;\mathrm{Lip}\left(W^{k,2},W^{k-i,2}\right)\right)$, i.e., there exists some constant $[\Q]_{\alpha }<\infty$
		such that
		\[
		|Q^i_{st}|_{\mathrm{Lip}(W^{k,p},W^{k-i,p})}\leq [\Q]_{\alpha }(t-s)^{i\alpha }\,,\quad \text{for every}\enskip (s,t)\in\Delta ;
		\]
		
		\item [($\star\star$)]\label{RD2}
		Chen's relations hold true, namely, for every $(s,\theta ,t)\in\Delta _2,$ we have
		\begin{equation}
		\label{chen}
		\delta Q^1_{s\theta t}=0\,,\quad \delta Q^2_{s\theta t}=(Q^1-Q^1(0))_{\theta t}\circ Q^1_{s\theta } \,,
		\end{equation}
		as operators acting on the scale $(W^{k,2})_{-3\leq k\leq 3}.$
	\end{enumerate}
\end{definition}

\begin{example}
	\label{example:URD}
	An example of such object is defined as the pair $\B=(B^1,B^2)$ via
	\[\begin{aligned}
	B^1_{st}:= X^i _{st}\partial _\mu 
	\\
	B^2_{st}:= X^i _{st}X^j _{st}\partial _{ij} + \LL^i _{st}\partial _i\,,
	\end{aligned}
	\]
	where the coefficients satisfy Assumption \ref{ass:geometric}.
	In this case $\B$ is in fact linear and this corresponds to the usual notion of an unbounded rough driver \cite{bailleul2017unbounded}.
\end{example}

The geometricity of the coefficient path allows the associated unbounded rough driver $\B$ to satisfy a bunch of convenient algebraic rules,
for instance regarding the non-commutative bracket $2B^2_{st}-(B^1_{st})^2.$
We summarize these in the next result, without proof (see nevertheless \cite[Appendix A]{hocquet2018ito}).
\begin{lemma}[Weak geometric relations]
	\label{lem:weak_geo}
	Let $\X=(X,\LL)$ be as in Assumption \ref{ass:geometric}, and define the (linear) unbounded rough driver $\B$ as in Example \ref{example:URD}.
	Introduce the generalized bracket\footnote{
		In \cite{hocquet2018ito}, the bracket was defined somewhat erroneously as half the above quantity, i.e.\ $B^2_{s,t}-\frac12B^1_{s,t}\circ B^1_{s,t}$.
		To fit with the usual convention (\cite{friz2014course}), we here adopt the notation \eqref{bracket}.
	}
	\begin{equation}
	\label{bracket}
	[\B]_{st}:=2B^2_{st}-B^1_{st}\circ B^1_{st}\,,\quad\text{for}\enskip  (s,t)\in\Delta .
	\end{equation} 
	The following holds.
	\begin{enumerate}[label=(\roman*)]
		\item For each $(s,t)\in\Delta ,$ $[\B]_{st}$ is a derivation.
		In fact $[\B]$ and $\LL$ are related through the formula
		\begin{equation}\label{formula}
		[\B]_{st}=\left(2\LL_{st}^i -X^\mu _{st}\partial _\mu X_{st}^i\right)\partial _i\,,\quad (s,t)\in\Delta .
		\end{equation}
		\item If $\div X=0,$ then $\div [\B]= 0.$
	\end{enumerate}
\end{lemma}

\begin{example}
	\label{example:bracket}
	Let $(Z^\mu,\mathbb Z^{\mu\nu})_{1\leq \mu,\nu\leq m}$ be an $m$-dimensional, two-step geometric rough path of H\"older regularity $\alpha>\frac13$ and assume that $\sigma=\sigma^{i}_\mu(x)$ belongs to $C^3_b(\mathbb T^d;\R^d\otimes \R^m)$.
	Define
	\[ 
	X^i_{st}:=\sigma^i_\mu Z^\mu_{st},\quad
	\LL_{st}^i:= \mathbb Z_{st}^{\mu,\nu} \partial_j\sigma^i_\mu \sigma^j_\nu\enskip.
	\]
	Then, it is easily checked that $\X=(X,\LL)$ satisfies the requirements of Assumption \ref{ass:geometric}.
	
	Moreover, denoting by $\mathbb A$ the \textit{L\'evy area} of $\Z$, formally
	\begin{equation}\label{levy}
	\mathbb A_{st}:=\frac12\iint_{s<r_1<r_2<t}(\dd Z_{r_1}\otimes \dd Z_{r_2} -\dd Z_{r_2}\otimes\dd Z_{r_1})\,, 
	\end{equation}
	one sees using geometricity that
	\[ 
	\LL_{st}^i:= (\frac12Z^\mu_{st} Z_{st}^\nu + \mathbb A^{\mu,\nu}_{st})\partial_j\sigma^i_\mu \sigma^j_\nu
	\]
	If we denote by $V^i\mapsto \widehat V:=V^i\partial_i$ the natural isomorphism between coefficients and vector fields, exploiting symmetries one sees that
	\[ 
	\mathbb{\widehat L}_{st}= \frac12(Z^\nu_{st}\widehat \sigma_\nu)\circ (Z^\mu _{st}\widehat \sigma _\mu) - \frac12X_{st}^iX^j_{st}\partial_{ij} + \frac12 \mathbb A_{st}^{\mu,\nu}[\widehat \sigma_\nu,\widehat \sigma_\mu]\,.
	\]
	where $[\cdot ,\cdot]$ is the usual Lie bracket for vector fields. Using the notations of Example \eqref{example:URD}, this means that the following identity holds
	\begin{equation}\label{id_bracket}
	[\B]_{st}\equiv 2B^2_{st}- (B^1_{st})^2 = \mathbb A^{\mu,\nu}_{st}[\widehat\sigma_\nu,\widehat \sigma_\mu].
	\end{equation} 
	
	If the vector fields $\widehat\sigma_\mu,\mu=1,\dots m,$ commute, the above relation tells us that $[\B]_{st}=0$. In that case, the formula \eqref{formula} permits to construct $\LL^i$ without further knowledge than the values of the path $X=\sigma_\mu Z^\mu$ itself (in particular, the solutions obtained in Theorem \ref{thm:existence} will not depend on $\mathbb A$). 
\end{example}

\subsection{Notion of solution and main results}
\label{subsec:results}
In the whole paper, we consider an ansatz equation of the form
\begin{equation}
\left\{\begin{aligned}
\label{rough_PDE_gene}
&\dd v_t=(-\partial _if^i_t +f^0_t)\dd t+\dd \Q_t^i(g_t),\quad 
\text{on}\enskip [0,T]\times \T^d,
\\
&v_0=v^0\in L^p(\T^d)\,,\quad p\in [1,\infty]\,,
\end{aligned}\right.
\end{equation}
where $g$ possibly depends on $v$ and $\Q$ is an (affine linear) unbounded rough driver.
Concerning the drift term, we will assume that $f^i,i=0,\dots d$ belongs to $L^p(0,T;L^p)$ (i.e.\ it is $p$-integrable as a mapping from $[0,T]$ into $L^{p}$), while the derivation $\partial _i=\frac{\partial }{\partial x_i}$ is understood in distributional sense.

\begin{definition}[$L^p$-solution]
	\label{def:solution}
	Let $T>0$, $\alpha \in(1/3,1/2]$ and fix $p,p'\in [1,\infty]$ with $1/p+1/p'=1, $
	and consider $g\in L^\infty(0,T; L^p),$ while $f^i\in L^1(0,T;L^p),$ $i=0,\dots d.$
	
	Assume the existence of a path $g'\colon [0,T]\to L^p$ such that letting
	$R^g_{st}:=g_{st}- Q^1_{st}(g'_s)$, 
	there are controls $\omega,\varpi $ and $L>0$ such that for each $(s,t)\in\Delta$ with $|t-s|\leq L,$
	\begin{equation}
	\label{controlled}
	|R^g_{st}|_{W^{-2,p}}\leq \omega (s,t)^{2\alpha }
	\quad \text{and}\quad |g'_{st}|_{W^{-1,p}}\leq \varpi(s,t)^\alpha\,.
	\end{equation}
	\begin{itemize}
		\item
		A path $v\colon[0,T]\to L^p$ is said to be an {\emph{$L^p$-weak solution}} to the
		rough PDE \eqref{rough_PDE_gene} if it fulfills the following conditions
		\begin{enumerate}[label=(\arabic*)]
			\item $v\colon[0,T]\to L^p$ is weakly-$*$ continuous and belongs to $L^\infty(0,T;L^p)$;
			\item for every $\phi \in W^{2,p'}$, and every $(s,t)\in\Delta :$
			\begin{equation}
			\label{nota:solution}
			\int_{\T^d} v_{st}\phi \dd x+\iint_{[s,t]\times\T^d} (f^i\partial _i\phi + f^0\phi ) \dd x\dd r
			=\Big\langle Q^1_{st}(g_s) + Q^2_{st}(g'_s)+v_{st}^\natural,\phi \Big\rangle\,,
			\end{equation}
			for some $v^\natural\in \mathcal C^{1+}(0,T;W^{-3,p}).$
		\end{enumerate}
		
		\item
		An $L^p$-solution $v$ is called an {\emph{$L^p$-energy solution}} of \eqref{rough_PDE_gene} if additionally
		\begin{equation}
		\label{Lp_sol} 
		v\in L^p(0,T; W^{1,p}).
		\end{equation} 
	\end{itemize}
\end{definition}

\begin{remark}
The first of the two conditions in \eqref{controlled} means that a cancellation occurs in $R^g$ and is conveniently alluded to by the expression `$g$ is controlled by $Q$ with Gubinelli derivative $g'$'.
The Banach space of such controlled rough paths will be introduced below in Definition \ref{def:controlled}, where it will be used to formulate the remainder estimates.
Note that \eqref{controlled} implies that $H_{st}\equiv Q^1_{st}(g_s) + Q^2_{st}(g'_s)$ is amenable to the hypotheses of the the Sewing Lemma (Theorem \ref{thm:sewing}), in the space $E=W^{-3,p},$ since 
\begin{multline*}
|\delta H_{s\theta t}|_{W^{-3,p}}= |- (Q^1-Q^1(0))_{\theta t}(R^{g}_{s\theta}) - (Q^{2}-Q^{2}(0))_{\theta t}(\delta g'_{s\theta})|_{W^{-3,p}} 
\\
\lesssim [\Q]_{\alpha}\left ((t-\theta)^\alpha\omega(s,\theta)^{2\alpha} + (t-\theta)^{2\alpha}\varpi(s,\theta)^{\alpha}\right ) \quad \in \CC(E)\,.
\end{multline*}
(see \eqref{content} below for more detailed computations).
Hence the existence and uniqueness of $v^{\natural}$ in \eqref{nota:solution}.
\end{remark}

We can now state our main results for the quasilinear Cauchy problem:
\begin{equation}
\label{problem_main}
\left\{
\begin{aligned}
&\dd u - \partial _i(a^{ij}(t,x,u)\partial _ju) = \dd \X^i\partial _i u\quad \text{on}\enskip [0,T]\times\T^d\,,
\\
&u_0\in L^2(\T^d)\,.
\end{aligned}\right.
\end{equation}

Since they involve different assumptions, we state existence and uniqueness in separate theorems.
The following existence result will be shown in Section \ref{sec:existence}.

\begin{theorem}[Existence of $L^2$-solutions]
	\label{thm:existence}
	Take $u_0\in L^2$ and let $a^{ij}$ be as in Assumption \ref{ass:A}. Moreover, let $\X=(X^{i},\LL^{i})_{i=1,\dots ,d}$ be as in Assumption \ref{ass:geometric}.
	There exists an $L^2$-energy solution to the problem \eqref{problem_main}.
\end{theorem}

A noticeable aspect of the previous theorem is that solutions exist in the same class as the one described in the linear theory \cite{hocquet2017energy}, in particular they are global in time. The regularity of the coefficients is precisely enough to make sense of the equation on $v:=u^2,$ as an $L^1$ solution of a similar problem, which in turn allows to obtain the key a priori estimate in the energy space $L^\infty(L^2)\cap L^2(W^{1,2})$. This estimate will be a consequence of a similar Gronwall-type argument as for the linear case.
The existence of $L^2$-solutions will then be shown by a compactness argument, using the fact that the driving coefficient path $\X$ is geometric.  We note that this strategy follows essentially the lines of \cite{hocquet2017energy,hocquet2018ito}.\smallskip

As for uniqueness of solutions, the situation becomes much more involved. Typically, one aims to look for an estimate on the difference of two solutions, which in general involves moments on the spatial derivatives. As is quicky realized, these bounds are stronger than those needed for the existence step. Namely, the $L^\infty(L^2)\cap L^2(W^{1,2})$-estimate on $u$ is no longer sufficient to conclude.
A basic strategy is then to take a lower exponent than for existence, by searching to prove an $L^\infty(L^1)$ estimate instead of $L^\infty(L^2)$. But even in that case, the ``naive'' approach to estimate the rough integral seems to inevitably make the essential supremum of the gradient of solutions appear, which is too strong for our purposes.
There is however a favorable case that can be treated at the level of integrability described in Theorem \ref{thm:existence}, which is when the driving coefficient path $X^i_t(x),i=1,\dots ,d,$ is divergence-free. In that case, the $L^\infty(L^1)$ estimate simplifies, which allows us to show the following.

\begin{theorem}[Uniqueness for divergence free vector fields]
	\label{thm:uniqueness}
	Assume that the hypotheses of Theorem \ref{thm:existence} are satisfied, and let the divergence of $X$ be zero for all times, i.e.\ assume that $\div X_t(\cdot )=0,$ $\forall t\in [0,T].$
	Then, the solution constructed above is unique in the class of $L^2$-energy solutions.
\end{theorem}

The key argument in the proof of the above uniqueness result is the so-called renormalization property, Theorem \ref{thm:renorm}, which states that Nemytskii operations of the form $u\mapsto \beta \circ u,$ for $\beta \in C^2 ,$ give rise to new solutions of a similar problem. As seen in Section \ref{sec:uniqueness}, combining this fact with a suitable approximation argument yields the possibility of estimating the $L^\infty(L^1)$-norm of the difference of two given solutions $u^1$ and $u^2$.
In the divergence-free scenario, this estimate is enough because, it is possible to integrate this inequality and to conclude thanks to the rough Gronwall estimate, Lemma \ref{lem:gronwall}. \smallskip

In the case when $\div X\neq 0$ however, the previous simple idea fails, unless the two solutions have bounded spatial derivatives. A uniqueness criterion based on the latter condition would not be very satisfactory since it is not known whether such solutions exist, at least in a rough scenario.
Following an idea of \cite{hocquet2018generalized}, a possibility is then to take the product of $v_t(x):=|u_t^1(x)-u^2_t(x)|$ (which formally solves some rough parabolic inequality) with a ``weight function'' $m_t(x)$,
solution of the backward dual equation
\begin{equation}
\label{m_intro}
\left\{
\begin{aligned}
&\dd m  + (Am -b^i\partial _i m)\dd t =(\dd \X^i\partial_i+\div\dd\X)m\,,\quad \text{on}\enskip [0,T]\times \T^d\,,
\\
&m_T(x)=1\,,
\end{aligned}\right.
\end{equation} 
for well-chosen elliptic operator $A_t$ and velocity $b^i_t(x).$
The main strategy is then to obtain the desired $L^\infty(L^1)$ estimate by showing that
the (unique) solution $m$ of \eqref{m_intro} is bounded below by a positive constant.

\begin{remark}
	In the setting of stochastic PDEs a similar duality method has been used before, leading to backward SPDEs of the form \eqref{m_intro}. See, e.g., \cite{coghi2019stochastic,debussche2020diffusion}.
\end{remark}

In Section \ref{sec:parabolic}, we will extend our previous results on rough parabolic equations \cite{hocquet2018ito} by showing (among other things) that for velocities that are subject to the Ladyzenskaja-Prodi-Serrin -type condition (LPS, in short) 
\[
b\in L^{2r}(0,T;L^{2q}),\quad
\text{where}\quad \frac{1}{r}+ \frac{d}{2q}<1\,,
\]
then for a slightly different class of rough evolution equations as \eqref{m_intro} (but ``close enough''), it is possible to show that
the corresponding solutions remain positive for short times. That $m$ shares the same property will follow under the condition that moment estimates (with arbitrary exponents) hold for $\Phi$ and $\nabla\Phi$, where $\Phi$ solves of the rough transport equation
\begin{equation}
\label{transport_intro}
\left\{\begin{aligned}
&\dd \Phi = -\dd\X^i\partial _i\Phi  - \div \dd\X\quad \text{on}\enskip [0,T]\times \T^d\,,
\\
&\Phi _0=0\,.
\end{aligned}\right.
\end{equation}
Solving \eqref{transport_intro} in $L^\infty$ requires new arguments in comparison with the ones used in \cite{bailleul2017unbounded} (see also \cite{catellier2016rough}), because of the absence of a maximum principle due to the additive rough input. It will be addressed in Section \ref{sec:transport} thanks to an affine-type generalization of the product formula of \cite{hocquet2018ito}. While moment bounds for $\Phi$ will be shown for any spatial dimension, it is not clear yet how to extend the argument to the case of a \textit{system} of rough transport equations. Consequently, we are able to obtain the desired moment estimates on $\nabla\Phi$ only when $d=1$. As a consequence, in the next results we focus on the one-dimensional torus $\mathbb T^1$ (the $d$-dimensional case with $d>1$ will be presented in a future contribution). For matters of readability, in the sequel we denote by $\partial_x $ the one-dimensional nabla operator, namely $\partial_x:=\frac{\partial}{\partial x}$.

\begin{theorem}[Uniqueness 2]
	\label{thm:regular}
	Set $d=1$ and let $\X$ satisfy the hypotheses of Theorem \ref{thm:existence}. Assume further that $\rho(\partial_x\X)<\infty$, i.e.\
	$\partial_x X\in C^\alpha (0,T;W^{3,\infty})$ and $\partial_x \mathbb L\in C_{2}^{2\alpha}(0,T;W^{2\infty}).$
	Let $u^1$ and $u^2$ be $L^2(\mathbb T^1)$-energy solutions of \eqref{problem_main} such that
	\begin{equation}\label{LPS_gradient}
	\partial_x u^i \in L^{2r}(0,T;L^{2q}),\quad i=1,2\,,
	\end{equation}
	where $r\in (1,\infty],$ $q\in[1,\infty],$ are given numbers subject to the LPS condition
	\begin{equation}
	\label{integ:intro}
	\frac{1}{r}+ \frac{1}{2q}<1\,.
	\end{equation} 
	
	If $u^1_0=u^2_0,$ then $u^1=u^2\,.$
\end{theorem}

A logical question to ask is whether solutions exist with such level of integrability. In dimension one, it is sufficient to obtain an $L^\infty(L^2)$ a priori estimate on the spatial derivative $v=\partial _xu$, which is easily seen to satisfy a similar equation as $u$ itself, provided $a$ is regular enough.
The only significative difference is the appearance of the following Burgers-type nonlinearity in the drift associated to the equation on $v$
\[
F(v)=-\partial _x\big(a_z(t,x,u)v^2 \big)\,.
\]
As a consequence, we obtain a Bihari-Lassalle type estimate that only guarantees a uniform bound over a short time $T_*\in(0,T]$, thereby proving \textit{local} existence in time, in the class described by \eqref{integ:intro}-\eqref{LPS_gradient}. (On the other hand, since the $L^2$ theory of Theorem \ref{thm:existence} gives a global solution, it is still possible to extend the unique solution $u$ on $[0,T_*)$ after $T_*,$ however the above result does not guarantee uniqueness after such extension).

\begin{theorem}[Existence and uniqueness in dimension one]
	\label{thm:existence_2}
	Fix $d=1$  and consider $\X$ as in Theorem \ref{thm:regular}. Assume that $a\colon[0,T]\times\R\times \R\to \R$ is $C^1_b$ with respect to the second and third variables, for every $t\in[0,T]$, and such that Assumption \ref{ass:A} holds. 
	Suppose furthermore that $u_0\in W^{1,2}(\mathbb T^1).$
	
	There exists $T_*\in(0,T]$, and a unique $L^2(\mathbb T^1)$-energy solution $u$ to \eqref{problem_main} up to time $t=T_*$ such that
	\begin{equation}
	\|u\|_{L^\infty(0,T_*;W^{1,2})}<\infty\,.
	\end{equation} 
\end{theorem}
The proof of this Theorem will be obtained in Section \ref{sec:higher}, where existence in the above class will be shown locally in time (the uniqueness part is simply a consequence of Theorem \ref{thm:regular}).

\subsection{Sketch of the proof of uniqueness when $d=1$ and $\partial_x X\neq0$}
\label{subsec:explication}
Since it is the main burden of the manuscript, we now explain briefly the main idea for the proof of Theorem \ref{thm:regular}.
If $u^1,u^2$ are two solutions of the same Cauchy problem, denote by $A_t$ the symmetric operator
\[
A_t\varphi:= \partial _x(a(t,x,u^1) \partial _x\varphi )\,,\quad \varphi \in W^{1,2}\,,
\]
and observe that $v:=u^1-u^2$ solves the equation
\[
\left\{\begin{aligned}
&\dd v = \big[Av + \partial _x(a(u^1)-a(u^2)\partial _xu^2)\big]\dd t+ \dd \B v
\\
&v_0=0\,,
\end{aligned}\right.
\]
where we use the shorthand notation $a(u)=a(t,x,u),$ and where $\B$ is the unbounded rough driver introduced in Example \ref{example:URD}.
Using the renormalization property, we have formally with $\beta (z):=|z|$, for any positive test function $\phi $:
\begin{multline*}
\int_{\T^d} \left(|v|_{st}\phi- (B^{1,*}_{st} + B^{2,*}_{st})\phi\right) \dd x
- \langle|v|^{\natural}_{st},\phi \rangle
\\
\leq
\int_s^t \big\langle A|v|,\phi \big\rangle\dd r
+\int_s^t\Big\langle \partial _i\left(\mathrm{sgn}v(a(u^1)-a(u^2))\partial _j u^2\right), \phi \Big\rangle
=\int_s^t \Big\langle \left[A+\Lambda \right]|v|, \phi \Big\rangle\dd r
\end{multline*}
where $\langle v^\natural,\phi \rangle$ is a remainder in $\CC(0,T;W^{-3,1})$ and
$\Lambda $ is the flux term defined by the linear operator
\[
\left[\begin{aligned}
&\Lambda \varphi= \partial _x(b_t(x) \varphi)\,,\quad \varphi \in W^{1,2}\,,
\\
&b_t(x):= \mathbf 1_{v_t(x)\neq 0}\frac{a(t,x,u^1_t(x))-a(t,x,u^2_t(x))}{v_{t}(x)}\partial _xu^2_t(x)\,.
\end{aligned}\right.
\]
But thanks to the fact that $a$ is Lipshitz, we observe that $b$ inherits the integrability of $\partial_x u^2,$ that is
\[
b\in\ L^{2r}(0,T;L^{2q})\quad \text{with}\enskip
\frac{1}{r} + \frac{1}{2q}<1\,.
\]

Next, if $m_t$ denotes a non-negative solution of the following generic, backward problem
\begin{equation}\label{generic}
\left\{
\begin{aligned}
&\dd m + \mathscr A_tm_t= -\dd \B^*m\,,\quad \text{on}\enskip [0,T]\times\T^d\,,
\\
&m_T(\cdot ):=m^T\in L^\infty(\T^d)\enskip \text{given,}
\end{aligned}\right.
\end{equation}
then the product formula, Proposition \ref{pro:product_backward}, asserts the existence of $\langle|v|,m\rangle^{\natural}\in\CC(0,T;\R)$ so that
\begin{equation}
\label{product_v_m}
\begin{aligned}
\int _{\T^d}(|v|m)_{st} \dd x 
&= \int_{\T^d}|v|_{st}m_t + \int_{\T^d}|v_s|m_{st}
\\
&\leq
\left\langle \left(\int_s^t\left[A+\Lambda \right]|v|\dd r\right) + (B^1_{st} + B^2_{st})|v_s|, m_t \right\rangle
\\
&\quad 
+\left\langle |v_s|,\left(\int_s^t\mathscr A m\dd r\right) - (B^{1,*}_{st} + B^{2,*}_{st})m_t\right\rangle
+ \big\langle|v|,m \big\rangle^{\natural}_{st}
\\
&=\int_s^t\Big\langle |v|,\left[A^*+\Lambda^* +\mathscr A \right]m\Big\rangle 
\end{aligned}
\end{equation} 
for every $0\leq s\leq t\leq T.$
In particular, if $m_T=1$ and $\mathscr A= -A - \Lambda ^*,$ i.e.\ if $m\geq 0$ is a solution to the backward dual equation \eqref{m_intro},
then the right hand side of \eqref{product_v_m} vanishes, and we end up with the following weighted inequality
\[
\int _{\T^d}|v_t|m_t \dd x \leq 0\,.
\]
The main difficulty is then to justify that $m$ is bounded below by a positive constant, from which uniqueness will follow. Namely, we aim to prove that
\begin{equation}\label{boundedness_below}
\inf_{t\in[T_+,T],x\in\T^d} m_t(x)>0
\end{equation}
for some $T_+\in[0,T).$ In \cite{hocquet2018generalized}, this was easily seen thanks to the fact that $m$ has a probabilistic representation, a tool that is missing here due to lower regularity of the coefficients.
Instead we have to rely on the Moser Iteration-type technique introduced in \cite{hocquet2018ito}, and generalized here to a broader context. 
Note that the adjoint rough term in \eqref{m_intro} involves a multiplicative noise with a term of zero order, which is precisely related to fact that $\div X\equiv \partial_x X\neq 0.$ Unfortunately, in that case Moser's recursive formula gives rise to a blowing-up term in the estimate on iterated powers of the solution, and hence one cannot infer the $L^\infty$-estimate as easily as in \cite{hocquet2018ito}. Instead, one has to rely on the fact that the solution $m_t(x)$ to \eqref{generic} equals $\exp(\check \Phi _t)(z+1)$ where $z_t(x)$ solves another parabolic equation with pure transport noise, while $\Phi _t(x):=\check\Phi _{T-t}(x)$ is an $L^\infty$-solution of the rough transport equation
\begin{equation}
\label{transport_intro_1D}
\left\{\begin{aligned}
&\dd \Phi = -\dd\X\partial _x\Phi  - \dd(\partial_x\X)\quad \text{on}\enskip [0,T]\times \T^1\,,
\\
&\Phi _0=0\,.
\end{aligned}\right.
\end{equation}
The desired property \eqref{boundedness_below} will then follow by combining the Moser-type estimate on $z$ together with the boundedness of moments of $\Phi $ and $\partial_x \Phi.$ These arguments will be detailed in Sections \ref{sec:parabolic}, \ref{sec:transport} and \ref{sec:uniqueness_2}.

\section{Preliminaries}
\label{sec:preliminaries}
In this section, we state key results that will be used in the sequel, among which is the \textit{renormalization property} for parabolic equations with transport rough input.
We then recall the ``rough Gronwall Lemma'' as was introduced in \cite{deya2016priori,hofmanova2016rough}, which in our context is a crucial tool to obtain a priori estimates.  Another essential result is that of the remainder estimates, Proposition \ref{pro:apriori}, which in comparison with \cite{deya2016priori} are extended to affine-linear rough families of partial differential operators, i.e.\ of the form given in Definition \ref{def:Q}, and which will be encountered in later sections.

In the sequel, we fix $d\geq1$ and assume the following.

\begin{assumption}[strong parabolicity]
	\label{ass:parabolic}
	We are given symmetric coefficients $(a^{ij})\colon[0,T]\times\mathbb T^d\to \R^{d\times d}$ which are measurable and such that
	\[
	\lambda|\xi |^2\leq 
	a^{ij}(t,x)\xi ^i\xi ^j
	\leq \lambda^{-1}|\xi |^2,\quad \text{for all}\enskip \xi \in \R^d.
	\]
	for some constant $\lambda >0.$
\end{assumption}

The main uniqueness results (Theorems \ref{thm:uniqueness} \& \ref{thm:regular}) will be partially based on the following crucial property.

\begin{theorem}[Renormalization property]
	\label{thm:renorm}
	Let $\X=(X^i,\LL^i)_{1\leq i\leq d}$ be such that Assumption \ref{ass:geometric} holds, and for $\varphi \in W^{1,2}$ define
	\[
	A_tu:=\partial _i(a^{ij}(t,\cdot )\partial _iu)\,,\quad t\in[0,T]\,.
	\]
	Let $u$ be an $L^2$-solution of the parabolic equation 
	\begin{equation}
	\label{parabolic}
	\left\{\begin{aligned}
	&\dd u= \big[Au +F\big]\dd t +\dd \X ^i\partial _iu,\quad \text{on}\enskip [0,T]\times\mathbb T^d,
	\\
	&u_0=u^0\in L^2\,,
	\end{aligned}\right.
	\end{equation}
	where $F\in L^2(0,T;W^{-1,2}),$ and $\partial _i=\frac{\partial }{\partial x_i},$ and where the coefficents $a^{ij}$ satisfy Assumption \ref{ass:parabolic}.
	
	Then, for every $\beta \in C^2(\R)$ with $|\beta '|_{L^\infty}+ |\beta ''|_{L^\infty}<\infty,$ it holds in the $L^1$-sense:
	\begin{equation}
	\label{chain_rule_thm}
	\quad 
	\left\{\begin{aligned}
	&\dd \beta \circ u = \Big[A\left(\beta \circ u\right) - (\beta ''\circ u) a^{ij}\partial _iu\partial _ju + (\beta'\circ u )F\Big] \dd t + \dd \X^i\partial_i \left(\beta \circ u\right),
	\\
	&\quad \quad \quad \quad \quad 
	\text{on}\enskip [0,T]\times\mathbb T^d,
	\\
	&\beta \circ u_0=\beta \circ u^0\quad \text{on}\enskip \mathbb T^d\,.
	\end{aligned}\right.
	\end{equation}
	
	More explicitly,
	the path $[0,T]\to L^1,$ $t\mapsto \beta (u_t):=\beta\circ u_t$ is well-defined, controlled by $Q=X\cdot \nabla $ in the sense of \eqref{controlled};
	moreover writing  $F=\partial_i f^i -f^0$ with $f\in L^2(L^2),$ it satisfies, for any $\phi \in W^{2,\infty}$ and $(s,t)\in\Delta ,$ the relation
	\begin{multline}
	\label{chain_rule_thm_expl}
	\left \langle \beta (u)_{st},\phi \right \rangle +\iint_{[s,t]\times \T^d} \Big[\beta ''(u)\partial _iua^{ij}\partial _ju_r\phi + \beta '(u)a^{ij}\partial _ju\partial _i\phi
	\\
	+\beta ''(u)\partial _iuf^i\phi
	+ \beta '(u)(f^i\partial _i\phi + f\phi )
	\Big]\dd x\dd r 
	\\
	=\int_{\T^d}\beta (u_s)(B^{1,*}_{st}+B^{2,*}_{st})\phi \dd x + \left \langle \beta (u)^{\natural}_{st},\phi \right \rangle
	\end{multline}
	where $\beta (u)^\natural\in\CC(0,T;W^{-3,1}),$ i.e.\ there exists $a>1$ and a control function $\omega $ such that for every $(s,t)\in\Delta $,
	$\sup_{|\phi |_{W^{3,\infty}}\leq 1}\langle \beta (u)^{\natural}_{st},\phi \rangle\leq \omega (s,t)^a.$
\end{theorem}

\begin{proof}
The only difference here with the property shown in \cite{hocquet2018ito} is that the full space $\R^d$ is replaced by the $d$-dimensional torus $\mathbb T^d,$ but the proof is carried out in the same way (noticing for instance that \cite[Prop.\ 4.1]{hocquet2018ito} is a local statement). Hence, we only sketch the main arguments and refer to the latter reference for details.  

If $u$ is bounded, it is possible to iterate products as in \eqref{concl:prod} to obtain the chain rule on polynomials.
By the uniform remainder estimates satisfied by elements of the parabolic class $\mathcal H^{\alpha,p}_B$ (see \cite[Section 4]{hocquet2018ito}), one can then conclude from a density argument that \eqref{chain_rule_thm_expl} holds for such $u$.
The formula is then extended to \emph{any} solution, thanks to the observation that a ``sufficiently large'' class of rough parabolic equations has bounded solutions (for instance, this is the case when $F$ above belongs to $L^r(0,T;L^q)$ with $\frac{1}{r}+\frac{d}{2q}<1$).
Boundedness among the latter class is proved by applying a rough Moser iteration lemma (\cite[Proposition 6.2]{hocquet2018ito}). See also Theorem 5.5 below.
\end{proof}

Besides the renormalization property, one of the core arguments that will we used repeatedly in this paper is a Gronwall-type lemma, well-adapted to incremental equations of the form \eqref{nota:solution}.
\begin{lemma}[Rough Gronwall]
	\label{lem:gronwall}
	Let $E\colon[0,T]\to \R_+$ be a path such that there exist constants $\kappa,L>0,$ a regular control $\omega ,$ and a superadditive map $\varphi $ with:
	\begin{equation}\label{rel:gron}
	E_{st}\leq \left(\sup_{s\leq r\leq t} E_r\right)\omega (s,t)^{\kappa }+\varphi (s,t),
	\end{equation}
	for every $(s,t)\in\Delta$ under the smallness condition $\omega (s,t)\leq L$.
	
	Then, there exists a constant $C _{\kappa ,L}>0$ such that
	\begin{equation}
	\label{concl:gron}
	\sup_{0\leq t\leq T}E_t\leq \exp\left(\frac{\omega (0,T)}{C _{\kappa ,L}}\right)\left[E_0+\sup_{0\leq t\leq T}\left|\varphi (0,t)\right|\right].
	\end{equation}
\end{lemma}

\begin{proof}
	See \cite{deya2016priori}.
\end{proof}

We are now interested in remainder estimates for equations of the form
\begin{equation}
\label{ansatz_Q}
\dd v = F\dd t + \dd \Q(v)\,,\quad \text{on}\enskip [0,T]\times \T^d
\end{equation} 
for distributional drifts $F\in L^p(0,T;W^{-2,p}),$ understood as the following Euler-Taylor expansion
\begin{equation}
\label{euler-taylor_Q}
v_{st} =\int_s^tF_r\dd r + Q^{1}_{st}(v_s) + Q^2_{st}(v_s)+ v^\natural_{st}\,,
\end{equation}
as an equality in $W^{-2,p}$ for every $(s,t)\in\Delta ,$ where the above integral is understood in the Bochner sense, 
and under the smallness condition that $v^\natural\in \CC(0,T;W^{-3,p}).$

Here $\Q:=(Q^1,Q^2)$ is by assumption an affine linear, unbounded rough driver as in Definition \ref{def:Q}.
The usual Chen's relations are replaced by the affine counterpart \eqref{chen}, which for the reader's convenience are rewritten here.
For every $(s,\theta ,t)\in\Delta _2$, it holds
\begin{equation}
\label{affine_chen}
\left[\begin{aligned}
&\delta Q^{1}_{s\theta t}=0
\\
&\delta Q^2_{s\theta t}= \left(Q^1-Q^1(0)\right)_{\theta t}\circ Q^1_{s\theta }
\end{aligned}\right.
\end{equation}
where $\circ$ denotes composition.

Prior to state the desired estimates, we now introduce what in our context plays the role of Gubinelli's controlled path space \cite{gubinelli2004controlling}, with respect to the first level of $\Q$, i.e.\ the $\alpha $-H\"older map 
\[t\mapsto Q_t:=Q^1_{0 t},\quad [0,T]\to \mathrm{Lip}(W^{k,p},W^{k-1,p}),\quad k=-2,-1,0.\]

\begin{definition}[Controlled paths]
	\label{def:controlled}
	Given $\alpha \in(1/3,1/2],$ we define the {\emph{controlled path space}} $\mathcal D_Q^{\alpha,p}$
	as the linear space of couples $(g,g')\in L^\infty(0,T;L^p)$ such that $g$ is controlled by $Q$ with Gubinelli derivative $g' $, by which it is understood that there exists a control $\omega_1 \colon\Delta \to \R_+$ such that for every $(s,t)\in\Delta$
	\begin{equation}
	\label{hyp:control_R_g}
	|R^g_{st}|_{W^{-2,p}}\leq
	\omega _1(s,t)^{2\alpha }\,,
	\end{equation}
	where
	\begin{equation}
	\label{R_g}
	R^g_{st}:=g_{st}-Q_{st}^1(g_s') .
	\end{equation}
	
	Moreover, we assume the existence of another control $\omega _2$ so that for each $(s,t)\in\Delta ,$
	\begin{equation}
	\label{hyp:control_dg}
	|g_{st}|_{W^{-1,p}}\leq \omega _2(s,t)^{\alpha }\,,
	\end{equation} 
	and similarly for $g'$.

	We shall denote by $\|R^g\|_{2\alpha ,-2,p}(s,t)$ (resp.\ $\|g\|_{\alpha ,-1,p}(s,t)$) the smallest possible right hand side for \eqref{hyp:control_R_g} (resp.\ \eqref{hyp:control_dg}).
	Equipped with the norm
	\begin{multline}
	\label{DQ_norm}
	\| (g,g')\|_{\mathcal D^{\alpha,p}_Q}:=\|g\|_{L^\infty(0,T;L^p)}+\|g'\|_{L^\infty(0,T;L^p)} 
	\\
	+ \|R^g\|_{2\alpha,-2,p}(0,T) + \|g\|_{\alpha ,-1,p}(0,T)+\|g'\|_{\alpha ,-1,p}(0,T)\,,
	\end{multline}
	the linear space $\mathcal D_Q^{\alpha,p}$ forms a Banach space.
\end{definition}

As is the case of standard controlled paths spaces \cite{friz2014course}, it should be observed that the previous notion only depends on the \textit{first level} $Q$ of $\Q,$ which in turn explains the notation $\mathcal D_Q^\alpha$.

The following statement is an affine-linear extension of the corresponding result of Deya, Gubinelli, Hofmanov\'a and Tindel \cite{deya2016priori}.

\begin{proposition}[Remainder estimates]
	\label{pro:apriori}
	Let $p\in[1,\infty]$ and let $v\in L^\infty(0,T;L^p)$ be a solution of \eqref{ansatz_Q}, in the sense of the Euler-Taylor expansion \eqref{euler-taylor_Q},
	for some $F\in L^p(0,T;W^{-2,p}).$
	
	There are constants $C,L>0$ depending only on $\alpha $
	such that for each $(s,t)\in\Delta $ subject to the smallness assumption 
	$|t-s|\leq L,$
	it holds the estimates
	\begin{equation}
	\label{estimate_remainder}
	|v^{\natural}_{st}|_{W^{-3,p}}
	\leq C[\Q]_\alpha  \left((t-s)^{3\alpha }\| v\|_{L^\infty(s,t;L^p)}
	+(t-s)^\alpha \int_s^t|F_{r}|_{W^{-2,p}}\dd r\right)\,,
	\end{equation}
	and
	\begin{equation}
	\label{estimate_remainder_2}
	|v^{\natural}_{st}|_{W^{-2,p}}
	\leq C[\Q]_\alpha  \left((t-s)^{2\alpha }\| v\|_{L^\infty(s,t;L^p)}
	+\int_s^t|F_{r}|_{W^{-2,p}}\dd r\right).
	\end{equation}
	
	Moreover, assuming that $\alpha \leq 1/2$ and denoting by
	\[R^v_{st}:= v_{st}-Q^1_{st}(v_s)\,,\]
	there is $0<L_0(\alpha ,[\Q]_\alpha ,\|F\|_{L^1(W^{-2,p})})\leq L$ such that for $(s,t)\in\Delta $ with $t-s\leq L_0,$ it holds the estimates
	\begin{equation}
	\label{estimate_remainder_3}
	\left[\begin{aligned}
	&|R^v_{st}|_{W^{-2,p}}
	\leq C[\Q]_\alpha  \left((t-s)^{2\alpha }\| v\|_{L^\infty(s,t;L^p)}
	+\int_s^t|F_{r}|_{W^{-2,p}}\dd r\right)
	\\
	&|v_{st}|_{W^{-1,p}}\leq C \left(\left[\int_s^t|F_r|_{W^{-2,p}}\dd r\right]^{\alpha } + (t-s)^\alpha \|v\|_{L^\infty(s,t;L^p)}\right)\,.
	\end{aligned}\right.
	\end{equation}
	In particular, $v$ is controlled by $Q$ with Gubinelli derivative $v'=v$, and the previous bounds can be summarized by writing
	\[
	\|v\|_{\mathcal D_Q^{\alpha ,p}}\leq C \left ([\Q]_\alpha ,\alpha \right )\left (\|v\|_{L^\infty(L^p)}+ \|F\|_{L^1(W^{-2,p})}\right )\enskip.
	\]
	
\end{proposition}

\begin{proof}
	Since the main idea is essentially contained in \cite{deya2016priori}, we content ourselves to show that applying $\delta$ to the remainder $v^{\natural}$ yields a similar expression as in \cite[Thm.~2.9, eq.~(2.24)]{deya2016priori}.
	
	Define $\widetilde Q^i$ to be the linear part of $Q^i,$ i.e.\ $\widetilde Q^i_{st}:=Q^i_{st}-Q^i_{st}(0),$ for $i=1,2.$
	Thanks to \eqref{affine_chen}, we have
	\begin{equation}\label{content}
	\begin{aligned}
	\delta v^\natural_{st} 
	&= Q^1_{\theta t}(u_\theta )-Q^1_{\theta t}(u_s) - \widetilde Q^1_{\theta t}\circ Q^1_{s\theta }(u_s) +
	Q^2_{\theta t}(u_\theta )-Q^2_{\theta t}(u_s)
	\\
	&=\widetilde Q^1_{\theta t}(u_{s\theta } -Q^1_{s\theta }(u_s)) +
	\widetilde Q^2_{\theta t}(u_{s\theta} )
	\end{aligned}
	\end{equation}
	for any $(s,\theta ,t)\in\Delta _2,$
	where we used the fact that $Q^1_{\theta t}$ is affine linear.
	The remainder of the proof follows exactly the same steps as that of the previous reference, and hence is omitted (see also \cite{hocquet2018ito} for an alternative proof).
\end{proof}

\section{Existence of $L^2$-solutions: proof of Theorem \ref{thm:existence}}
\label{sec:existence}

Let $X(n)\in C^1(0,T;W^{3,\infty})$ such that the canonical lift $\X(n)\equiv S_2(X(n))$ converges to $\X$ for the metric defined in \eqref{rho_alpha}.
By classical results on quasilinear equations \cite[Chapter 5]{ladyzhenskaya1968linear}, and since our notion of solution encompasses the usual one in the case of smooth coefficients (as seen for instance in \cite[Section~4]{hocquet2017energy}), we see that there exists a unique solution $u(n)\in L^\infty(0,T;L^2)\cap L^2(0,T;W^{1,2})$ to the problem
\begin{equation}
\label{equation_un}
\left\{\begin{aligned}
&\dd u(n) - \partial _i\big(a^{ij}(t,x,u(n))\partial _j u(n)\big)\dd t=\dd \X^i(n)\partial _iu(n),
\\
&u_0(n)=u_0\,,
\end{aligned}\right.
\end{equation}
in the sense of Definition \ref{def:solution}.
Moreover, using the renormalization property (Theorem \ref{thm:renorm}) with the function $\beta (z)=z^2,$
we obtain that for every $\phi \in W^{3,\infty}$ and $(s,t)\in\Delta :$
\begin{multline}
\label{square}
\int_{\mathbb T^d}u^2(n)_{st}\phi \dd x 
\\
+ 2\iint_{[s,t]\times\mathbb T^d}\Big[\partial _ju(n) a^{ij}(t,x,u(n)) \partial _iu(n)\phi
+u(n)a^{ij}(t,x,u(n))\partial _ju(n)\partial_i\phi \Big]\dd x\dd t 
\\
= \int_{\T^d} u^2(n)(B^{1}_{st}(n) + B^2_{st}(n))^*\phi\dd x + \langle u_{st}^{2,\natural}(n),\phi \rangle,
\end{multline}
where $u_{st}^{2,\natural}(n)$ denotes some remainder in $\CC(0,T;W^{-3,1}).$
Making the choice $\phi =1,$ and then estimating the term $\langle u_{st}^{2,\natural}(n),1 \rangle$ by the $W^{-3,1}$-norm of the remainder, we obtain thanks to Proposition \ref{pro:apriori} and the bound below for $a^{ij}$:
\begin{multline}
\label{En}
E_t(n)-E_s(n)\leq
C(T,\lambda ,\rho_\alpha(\X))(t-s)^{\alpha }|u_s(n)|_{L^2}^2 
\\
+ C(\lambda ,\rho_\alpha(\X))(t-s)^{\alpha }\int_s^t(|u(n)|_{L^2}^2 +|\nabla u(n)|^2)\dd r,
\end{multline}
where we let for convenience $E_t(n):=|u_t|_{L^2}^2+\int_0^t|\nabla u(n)|_{L^2}^2\dd r.$
For $|t-s|$ small enough (depending on $\rho_\alpha(\X),\lambda $ but not on $n$) we can absorb the last term to the left, yielding \eqref{rel:gron} with $\kappa =\alpha $ and $\varphi=0.$ In particular, thanks to Lemma \ref{lem:gronwall} we obtain that
\begin{equation}
\label{bd:existence_1}
\sup_{n\in\mathbb N} \|u(n)\|_{L^\infty(L^2)} + \int_0^T|\nabla u(n)|_{L^2}^2\dd t\leq C|u_0|_{L^2}^2.
\end{equation}
Using \eqref{bd:existence_1} in \eqref{En}, we also obtain the uniform equicontinuity of $E_{st}(n),$ in the sense that for any $\epsilon >0,$ there exists $\delta >0$ such that for every $n\geq 0:$
\begin{equation}
|t-s|\leq \delta \Rightarrow |E_{st}(n)|\leq \epsilon .                                                              \end{equation}
The same is true for the 2-parameter quantity $|u_t-u_s|_{W^{-1,2}},$ $(s,t)\in\Delta$. Indeed, by Proposition \ref{pro:apriori}-\eqref{estimate_remainder_3}, and since $E_{st}(n)$ is uniformy bounded, for all $n\geq 0$ we have for $|t-s|\leq L_0$ small enough:
\[
|u_t(n)-u_s(n)|_{W^{-1,2}}^{\frac1\alpha }\leq C_\alpha \left(E_{st}(n) + (t-s)\sup_{[s,t]}E_{\cdot }(n)\right)=:\omega (n;s,t)\,.
\]
Using the uniform equicontinuity for $E(n),$ we therefore obtain the same property for the family of controls $\omega (n;\cdot ,\cdot )$
and thereby for $u(n)$ in $W^{-1,2}$, which proves our assertion.

Now, from the Banach Alaoglu Theorem \eqref{bd:existence_1} and Ascoli, we obtain a limit point $u\in L^\infty (L^2)\cap L^2(W^{1,2})\cap C(W^{-1,2})$ such that for a subsequence $u(n_k),k\geq 0,$
\begin{align}
\label{cv:existence}
u(n_k)\to u\quad \text{weakly in}\enskip L^2(0,T;W^{1,2}),
\\
\label{cv:existence2}
u(n_k)\to u\quad \text{strongly in}\enskip C(0,T;W^{-1,2}),
\intertext{and interpolating:}
\label{cv:existence3}
u(n_k)\to u\quad\text{strongly in}\enskip L^2([0,T]\times \T^d),
\\
\label{cv:existence4}
u(n_k)\to u\quad \text{almost everywhere.}
\end{align}
Because of \eqref{bd:existence_1}, the drift term
\[
\phi\in W^{1,2} \mapsto -\iint_{[s,t]\times\T^d} a^{ij}(r,x,u_r(n_k)) \partial_j u_r(n_k)\partial_i \phi \dd x\dd r
\]
is uniformly bounded in $W^{-1,2}$, and so is $|u(n_k)_{st}^\natural|_{W^{-3,2}}$ by Proposition \ref{pro:apriori}.
Therefore, there exists an element of $\CC(0,T;W^{-3,2})$ denoted by $u^{\natural}$ so that up to another (but identically labelled) subsequence $u(n_k)$:
\begin{equation}
\label{cv:2}
\langle u^{\natural}(n_k)_{st},\phi \rangle
\to
\langle u^{\natural}_{st},\phi \rangle.
\end{equation} 
For simplicity, denote by $\B(n_k)=\B\{\X(n_k)\}$ the unbounded rough driver of Example \ref{example:URD}. For $k\geq 0$, we have for every $\phi \in W^{3,2}$ and $(s,t)\in\Delta$
\begin{multline}
\label{eq:nk}
\langle u(n_k)_{st},\phi \rangle
=-\iint_{[s,t]\times\T^d} a(r,x,u_r(n_k)) \partial_i u_r(n_k)\partial_j \phi \dd x\dd r
\\
+\int_{\T^d}u_s(n_k)(x)(B^{1,*}_{st}(n_k) + B^{2,*}_{st}(n_k))\phi (x)\dd x
+\langle u^\natural_{st}(n_k),\phi \rangle.
\end{multline}
Thanks to \eqref{cv:existence}, \eqref{cv:existence2}, \eqref{cv:2}, \eqref{cv:existence4} and dominated convergence, we can take the limit in each term of \eqref{eq:nk}, which leads to 
\[
\langle u_{st},\phi \rangle
=-\iint_{[s,t]\times \T^d}a^{ij}(r,x,u)\partial _ju\partial _j\phi \dd x\dd r
+\int_{\T^d}u_s(B_{st}^{1,*} +B_{st}^{2,*})\phi \dd x
+\langle u^{\natural}_{st},\phi \rangle.
\]
This shows that $u$ is a solution, thus proving existence and Theorem \ref{thm:existence}.
\hfill\qed

\section{Uniqueness for divergence-free vector fields: proof of Theorem \ref{thm:uniqueness}}
\label{sec:uniqueness}

The proof is inspired from an idea of \cite{hofmanova2017quasilinear}, in a stochastic context.
Consider two solutions $u^1$ and $u^2$ and let $v:=u^1-u^2.$
We cannot estimate the $L^1$-norm of $v$ directly because the map $x\mapsto|x|$ is singular at $x=0,$
however we can define an approximation of it as follows.

Let $1>a_1>a_2>\dots>a_n>\dots>0$ be a decreasing sequence of numbers such that
\begin{equation}
\label{sequence_an}
\int_{a_1}^1\frac{\dd \theta }{\theta }=1,\enskip 
\int_{a_2}^{a_1}\frac{\dd \theta }{\theta }=2,\enskip \dots\enskip ,
\int_{a_n}^{a_{n-1}}\frac{\dd \theta }{\theta }=n.
\end{equation} 
For $n\geq 1,$ we let $\varrho _n(\theta )$ be a continuous function supported in $(a_n,a_{n-1})$ and such that
\[
0\leq \varrho _n(\theta )\leq \frac{2}{n\theta },\quad 
\]
and integrating to one, i.e.\
\[
\int_{a_n}^{a_{n-1}}\varrho _n(\theta )\dd \theta =1.
\]
We then define 
\[
\beta _n(x):=\int_0^{|x|}\dd y\int_0^y\varrho _n(\theta )\dd \theta ,
\]
so that in particular $\beta\in C^2(\R)$ and has bounded first and second order derivatives. Moreover, we have the estimates
\begin{equation}
\label{derivees}
|\beta '_n(x)|\leq 1\quad \text{and}\quad 0\leq \beta _n''(x)\leq \frac{2}{n|x|}.
\end{equation}

Since $u^1$ and $u^2$ are solutions, it holds in the $L^2$-sense:
\begin{equation}
\label{eq:v}
\left\{
\begin{aligned}
&\dd v= \Big[\partial_i\big(a^{ij}(u^1)\nabla v\big) +\partial_i\big(a^{ij}(u^1)- a^{ij}(u^2))\partial_ju^2\big)\Big]\dd t + \dd \B v,\quad 
\\
&v_0=0,
\end{aligned}\right.
\end{equation}
where for simplicity in the notations we abbreviate the term $a(t,x,u^i)$ by $a(u^i),$ $i=1,2$ and $\B=\B\{\X\}$ is as in Example \ref{example:URD}.
Thanks to \eqref{coercivity}, the equation \eqref{eq:v} is strongly parabolic, i.e.\ of the form 
\[
\dd v=(\partial _i(a^{ij}(t,x)\partial _j v) +f)\dd t+\dd \B v
\]
with $(a^{ij})$ bounded above and below, and $f=\div\big((a(u^1)- a(u^2))\nabla u^2\big)\in L^2(W^{-1,2}).$
Hence, using the renormalization property (Theorem \ref{thm:renorm}) for $v,$
we have in the $L^1$-sense:
\begin{equation}
\label{eq:beta_n}
\left\{
\begin{aligned}
&\dd \beta _n(v)
= \beta _n'(v)\Big(\partial _i(a^{ij}(u^1)\partial _jv) +\partial _i\big[(a^{ij}(u^1)-a^{ij}(u^2))\partial _j u^2\big]\Big)\dd t + \dd \B \left(\beta _n(v)\right),
\\
&\beta (v)_0=0.
\end{aligned}\right.
\end{equation} 
Testing the above against $\phi =1,$ we obtain
\begin{multline*}
\Big(\int_{\mathbb T^d}\beta _n(v)\dd x\Big)_{st}
+\iint_{[s,t]\times\mathbb T^d} \beta _n''(v)\partial _iva^{ij}(u^1)\partial _jv\dd x\dd r
\\
=-\iint_{[s,t]\times\mathbb T^d}\beta _n''(v) \partial _i v(a^{ij}(u^1)-a^{ij}(u^2))\partial _j u^2\dd x\dd r
+\big\langle (B ^1_{st}+B ^2_{st})\beta _n(v),1\big\rangle 
+ \langle\beta _n(v)^{\natural}_{st},1\rangle.
\end{multline*}
Using the bound below for $a^{ij},$
as well as \eqref{derivees} and the fact that $a^{ij}$ is Lipshitz with respect to the third variable,
we get the estimate
\begin{multline}
\label{pre_gron}
\Big(\int_{\mathbb T^d}\beta _n(v)\dd x\Big)_{st}
+\lambda\iint_{[s,t]\times\T^d} \beta _n''(v)|\nabla v|^2\dd x\dd r
\\
\leq 
\frac{C}{n}\iint_{[s,t]\times\T^d} |\nabla v||\nabla u^2|\dd x\dd r
+C\rho_\alpha(\X)  (t-s)^\alpha |\beta _n(v_s)|_{L^1} + \langle\beta _n(v)^{\natural}_{st},1\rangle.
\end{multline}
Next, we show that the remainder $R_{st}:=\langle\beta _n(v)^{\natural}_{st},1\rangle$ is zero.
Indeed one can write, using Chen's relations \eqref{chen}:
\[\begin{aligned}
\delta R_{s\theta t}
&= \big\langle B^1_{\theta t}[\beta _n(v)_{s \theta }-B^1_{s\theta }\beta _n(v_s)],1\big\rangle 
+ \big\langle B^2_{\theta t}(\beta _n(v))_{s \theta },1 \big\rangle
\\
&= \big\langle \beta _n(v)_{s \theta }-B^1_{s\theta }\beta _n(v_s),B^{1,*}_{\theta t}1\big\rangle 
+ \big\langle \beta _n(v)_{s \theta },B^{2,*}_{\theta t}1 \big\rangle
\end{aligned}
\]
But $\div X= 0,$ hence:
\[B^{1,*}_{st}1=-X_{st}\cdot \nabla 1=0\,.\]
On the other hand, since $\X$ is geometric, then Lemma \ref{lem:weak_geo} asserts that $\div [\B]=0,$ where we recall that $[\B]$ is the bracket $2B^2_{st}-(B^1_{st})^2.$
Therefore,
\[
B_{st}^{2,*}1=\frac12\left (B_{st}^{1,*}\circ B_{st}^{1,*}1 +[\B]_{st}^*1\right )=0,
\]
which shows that $\delta R_{s\theta t}=0.$ Consequently, $R_{st}$ is an increment, and since it belongs to $\CC(0,T;\R),$ it follows by standard arguments that $R_{st}=0.$

Going back to \eqref{pre_gron}, this gives the estimate
\begin{multline*}
\Big(\int_{\mathbb T^d}\beta _n(v)\dd x\Big)_{st}
+\frac{\lambda}{2}\iint_{[s,t]\times\T^d} \beta _n''(v)|\nabla v|^2\dd x\dd r
\\
\leq 
\frac{C'}{n}\iint_{[s,t]\times\T^d} \big[|\nabla u^1|^2+|\nabla u^2|^2\big]\dd x\dd r
+C'\rho_\alpha(\X)  (t-s)^\alpha \|\beta _n(v)\|_{L^\infty(s,t;L^1)} \,.
\end{multline*}
The rough Gronwall lemma (Lemma \ref{lem:gronwall}) now implies
\[
\sup _{r\in[0,T]}\int_{\mathbb T^d}\beta _n(v_r)\dd x
\leq \frac{C}{n}\exp\left(\frac{\rho_\alpha(\X)  T}{C_{L,\alpha }}\right)
\int_{0}^T \big[|\nabla u^1|_{L^2}^2+|\nabla u^2|_{L^2}^2\big]\dd r\,.
\]
As $n\to\infty,$ the right hand side goes to $0,$ from which we infer by monotone convergence:
\[
\sup_{r\in [0,T]}\int_{\mathbb T^d}|v(x)|\dd x=0,
\]
yielding uniqueness. This proves Theorem \ref{thm:uniqueness}.
\hfill$\qed$

\section{Rough parabolic equations}
\label{sec:parabolic}

The purpose of this section is to investigate existence, uniqueness and $L^\infty$  estimates for a class of non-degenerate parabolic equations with affine-linear rough input, with the aim of
providing the lower bound estimate which is needed in the proof of Theorem \ref{thm:regular}, but also completing the results of \cite{hocquet2017energy,hocquet2018generalized,hocquet2018ito}.
The last subsection also introduces a notion of solution for backward problems, and gives the proof of a backward versus forward product formula. This will also provide a rigorous justification of the ``change of weight'' operated in the proof of the second uniqueness result, i.e.\ formula \eqref{product_v_m}. See Section \ref{sec:uniqueness_2}.\smallskip

In the sequel, we suppose that we are given some $(d+2)$-dimensional path $Y^i\in C^\alpha(0,T;W^{3,\infty}),$ $i=-1,0,\dots d,$ as well as an ``enhancement'' $\Y$ of it (see Assumption \ref{ass:solution_coef} below), and we assume as before $\alpha>1/3.$
We are interested in existence, uniqueness and boundedness results for the following parabolic ansatz with unknown $z_t(x)$
\begin{align}
\label{parabolic_gene}
&\dd z_t +\left(-\mathscr M_tz_t + f^0 - \partial _if^i\right)\dd t = \dd (\Y_t^i\partial _i  + \Y^0_t) z + \dd\Y^{-1}\,,
\quad \text{on}\enskip [0,T)\times \T^d\,,
\intertext{where the initial datum $z_0=z^0\in L^2(\T^d)$ is given and}
\label{general_M}
&\mathscr M_tz=\partial _i\left(a^{ij}(t,x)\partial _jz +F^i(t,x)z\right) -b^i(t,x)\partial _iz - c(t,x)z\,.
\end{align}
Integrated, the drift term is to be understood as the following distribution for each $\phi \in W^{1,2}(\mathbb T^d)$
\begin{equation}
\begin{aligned}
\left\langle \int_s^t(-\mathscr Mz +f^0-\partial _if^i)\dd r,\phi \right\rangle
&:=\iint_{[s,t]\times \T^d} \Big[a^{ij}\partial _iz \partial _j\phi +F^iz\partial _i\phi + b^i(\partial _iz)\phi
\\
&\quad \quad \quad \quad \quad 
+ cz\phi  +f^0\phi +f^i\partial _i\phi  \Big]\dd r\dd x\enskip.
\end{aligned}
\end{equation}

Clearly, further information than the values of the path $Y$ itself is needed in order to make sense of the equation.
However, due to the presence of the additional free term $Y^{-1}$ (in constrast with what was encountered in earlier sections), it turns out that the nature of the information needed to enhance $Y$ is slightly different from before.
To see this we can integrate \eqref{parabolic_gene} first when $t\mapsto Y_t(\cdot )$ has finite variation. This yields the (formal) Euler-Taylor expansion:
\[\begin{aligned}
z_{st}
&-\int_s^t(\mathscr Mz+\partial_if^i -f^0)\dd r
\\
&= (Y^i_{st}\partial _i + Y^0_{st})z_s + Y^{-1}_{st} 
\\
&\quad \quad \quad \quad 
+ \int_s^t(\dd Y^i_r\partial _i+Y^0_r)\Big[\int_s^r (\dd Y^j_{\tau }\partial _j+\dd Y^0_{\tau })z_\tau + Y^{-1}_{s,r}\Big] + o(t-s)
\\
&= (Y^i_{st}\partial _i + Y^0_{st})z_s + Y^{-1}_{st} 
\\
&\quad 
+ \Big(\frac12Y^i_{st}Y^j_{st}\partial _{ij} + \left [\bb{Y^i}{\partial _iY^j}_{st}+Y^0_{st}Y^j_{st}\right ]\partial _j + \bb{Y^i}{\partial _iY^0}_{st}
+\frac12(Y^0_{st})^2\Big)z_s 
\\
&\quad \quad \quad \quad 
+ \bb{Y^i}{\partial _iY^{-1}}_{st} + \bb{Y^0}{Y^{-1}}_{st}
+z^{\natural}_{st}
\end{aligned}
\]
where from now on, for two generic paths $A_t(x)$ and $B_t(x)$ of finite variation, we adopt the suggestive notation
\begin{equation}
\label{suggestive}
\bb{A}{B}_{st}:=\int_s^t\dd A_rB_{sr}\,,\quad (s,t)\in \Delta .
\end{equation} 
Hence, we end up with the expansion
\[
z_{st}-\int_s^t(\mathscr Mz + \partial_if^i -f^0)\dd r = Q^1_{st}(z_s) + Q^2_{st}(z_s) + z^{\natural}_{st}
\]
where for $z\in L^2(\T^d)$, we introduce
\begin{equation}
\label{Q_Y}
\left[
\begin{aligned}
&Q^1_{st}(z ):= (Y^i_{st}\partial _i + Y^0_{st})z  + Y^{-1}_{st} 
\\
&Q^2_{st}(z ):=
\Big(\frac12Y^i_{st}Y^j_{st}\partial _{ij} + \left [\bb{Y^i}{\partial _iY^j}_{st}+Y^0_{st}Y^j_{st}\right ]\partial _j + \bb{Y^i}{\partial _iY^0}_{st}
+\frac12(Y^0_{st})^2\Big)z 
\\
&\quad \quad \quad \quad  
\quad \quad \quad 
+ \bb{Y^i}{\partial _iY^{-1}}_{st} + \bb{Y^0}{Y^{-1}}_{st}\,.
\end{aligned}\right.
\end{equation}

This discussion naturally leads us to the following assumption on the coefficients.

\begin{assumption}[geometricity, full ansatz]
	\label{ass:solution_coef}
	We assume that $\Y$ is a enhancement of the coefficient path $Y,$ which takes the form of a triad
	\begin{equation}\label{triad_generic}
	\Y=\Big([Y^i]_{i=-1,0,\dots ,d}\,;\,[\bb{Y^j}{\partial _j Y^i}]_{i=0,\dots ,d}\,;\,\bb{Y^j}{\partial _j Y^{-1}}+\bb{Y^0}{Y^{-1}}\Big)
	\end{equation}
	in the space
	\[
	\mathscr C^\alpha :=C^\alpha (W^{3,\infty})^{d+2}\times C_2^{2\alpha }(W^{2,\infty})^{d+1}\times C_2^{2\alpha}(W^{1,\infty})\,,\]
	and such that
	\begin{enumerate}[label=(\arabic*)]
		\item The following Chen's-type relations hold: for every $(s,\theta,t)\in\Delta_2$:
		\begin{equation}\label{gene_chen}
		\left\{\begin{aligned}
		&\delta Y^i_{s\theta t}=0\,,\qquad \text{for}\enskip i=-1,\dots,d
		\\
		&\delta \bb{Y^j}{\partial _jY}_{s\theta t}= Y_{\theta t}^j\partial _jY_{s\theta },\qquad \text{for}\enskip  i=0,\dots d,
		\\
		&\delta \left(\bb{Y^j}{\partial _jY^{-1}}+\bb{Y^0}{Y^{-1}}\right)_{s\theta t}
		= Y^j_{s\theta}\partial _jY^{-1}_{\theta t} + Y^0_{\theta t}Y^{-1}_{s\theta }\,.
		\end{aligned}\right.
		\end{equation}
		\item $\Y$ is \emph{geometric}, in the sense that there is a sequence 
		\[\begin{aligned}
		\Y(n)
		&:=\Big([Y^i(n)]_{i=-1,0\dots d}\,;\,[\bb{Y^j(n)}{\partial _jY(n)}]_{i=0\dots d}\,;
		\\
		&\quad \quad \quad \quad 
		\quad \quad 
		\bb{Y^j(n)}{\partial _jY^{-1}(n)}+\bb{Y^0(n)}{Y^{-1}(n)}\Big)
		\end{aligned}
		\]
		of canonical lifts of paths $Y(n)$ of finite variation in $W^{3,\infty},$ such that the bracket terms correspond to the integrals \eqref{suggestive}, and with the property that $\Y(n)\to \Y$ for the natural product topology of $\mathscr C^\alpha .$
	\end{enumerate}
\end{assumption}

\begin{remark}
The notation $\bb{Y^j}{\partial_j Y^{-1}}+\bb{Y^0}{Y^{-1}}$ for the third ``symbol'' in \eqref{triad_generic} might appear somewhat odd, since it suggests that the two terms may be defined separately. This is neither true in general (as seen for instance in the enhancement \eqref{odd_enhancement}), nor needed in the analysis. However, for matters of readability we prefer to stick to this rather suggestive notation.
\end{remark}

Our main result in this subsection is the following.

\begin{theorem}
	\label{thm:parabolic_gene}
	Let $\mathscr M$ be as in \eqref{general_M} such that
	$a^{ij}$ satisfies \eqref{ass:parabolic} while
	\begin{equation}
	\label{hyp:coeff}
	F^i,b^i,f^i\in L^{2r}(0,T;L^{2q})\,,\quad 
	\enskip c,f^0\in L^r(0,T;L^q),
	\end{equation} 
	for some numbers $r,q\in[1,\infty]$ subject to the condition
	\begin{equation}
	\label{subject_r_q}
	\frac{1}{r} +\frac{d}{2q}\leq 1\,.
	\end{equation} 
	Let $\mathbf Y$ be as in Assumption \ref{ass:solution_coef}.
	The following holds true.
	
	\begin{itemize}
		\item The two parameter family of affine-linear differential operators $\Q=(Q^1,Q^2)$ defined in \eqref{Q_Y} satisfies the axioms of Definition \ref{def:Q}.
		
		\item
		There exists a unique $L^2$-solution to \eqref{parabolic_gene}, understood in the sense of Definition \ref{def:solution} with $\Q$ as above.
		Moreover, the solution map
		\[\begin{aligned}
		\left(L^{2r}(0,T;L^{2q})\right)^3\times \left(L^r(0,T;L^q)\right)^2
		&\longrightarrow L^\infty(0,T;L^2)\cap L^2(0,T;W^{1,2})
		\\
		(F^i,b^i,f^i;c,f)
		&\longmapsto u
		\end{aligned}
		\]
		is continuous with respect to the strong topologies.
	\end{itemize}
\end{theorem}

We recall the following classical interpolation inequality (see \cite{ladyzhenskaya1968linear}).

\begin{proposition}
	\label{pro:interp}
	For every $\varrho, \sigma $ with the properties
	\begin{equation}\label{conditions}
	\frac1\varrho +\frac{d}{2\sigma }\geq \frac d4\quad\text{and}\quad \left[\begin{aligned}
	&\varrho \in[2,\infty]\,, \quad \sigma \in[2,\tfrac{2d}{d-2}]\quad\text{for}\enskip d>2
	\\
	&\varrho \in(2,\infty]\,,\quad \sigma \in[2,\infty)\quad \text{for}\enskip d=2
	\\
	&\varrho \in[4,\infty]\,,\quad \sigma \in[2,\infty]\quad \text{for}\enskip d=1\,.
	\end{aligned}\right.
	\end{equation}
	there is a constant $C_{\varrho ,\sigma }>0 $ such that for any $f\in L^\infty(0,T;L^2)\cap L^2(0,T;W^{1,2})$:
	\begin{equation}
	\label{interpolation_inequality}
	\|f\|_{L^\varrho (0,T;L^\sigma )}\leq C_{\varrho ,\sigma } \left(\|\nabla f\|_{L^2(0,T;L^2)}+\|f_r\|_{L^\infty(0,T;L^2)}\right)\,.
	\end{equation}
\end{proposition}

As an easy consequence of this inequality, we see that if $r,q\in[1,\infty]$ are numbers satisfying the condition \eqref{subject_r_q},
then it holds the estimate:
\begin{equation}
\label{interp_u}
\|u\|_{L^{\frac{2r}{r-1}}(L^{\frac{2q}{q-1}})}\leq C_{r,q} \left(\|\nabla f\|_{L^2(0,T;L^2)}+\|f_r\|_{L^\infty(0,T;L^2)}\right)\,.
\end{equation}

\subsection{Solvability}

First, observe that if $z$ is an $L^2$-solution of \eqref{parabolic_gene}, then by the product formula (Proposition \ref{pro:product})
\begin{equation}
\label{ito_square_free}
\dd z ^2 -2z \mathscr Mz\dd t=  2zf\dd t + (\dd \Y^i\partial _i+ 2\dd \Y^0)(z^2) + 2z\dd \Y^{-1} ,
\end{equation} 
in the $L^1$-sense.
Our aim now is to test against $\phi =1,$ and then apply the rough Gronwall estimate, Lemma \ref{lem:gronwall}. For that purpose, we first need to estimate the $1$-variation norm in the space $W^{-1,1}$ of the drift in \eqref{ito_square_free}, which for every $\phi\in W^{1,\infty}$ is given by
\[\begin{aligned}
\langle\mathscr D_{st},\phi \rangle
&=\int_s^t\langle -2z \mathscr Mz,\phi \rangle\dd r
\\
&:=2\iint_{[s,t]\times \T^d} \Big[a^{ij}\partial _iz(\partial _jz\phi  + z\partial _j\phi ) +F^iz(\partial _iz\phi + z\partial _i\phi ) 
\\
&\quad \quad \quad \quad \quad 
+ b^i(\partial _iz)z\phi + cz^2\phi  +f^0z\phi +f^i(\partial _iz\phi+ z\partial _i\phi ) \Big]\dd r\dd x\,.
\end{aligned}
\]
But immediate computations using H\"older Inequality yield:
\begin{equation}
\label{bd:reduced_drift}
\begin{aligned}
\frac12\Big| \mathscr D_{st}\Big|_{W^{-1,1}}
&\leq \lambda^{-1}(\|\nabla z \|^2_{2,2}+\|z \nabla z \|_{1,1})
+\|F^i,b\|_{2r,2q}\|\nabla z \|_{2,2}\|z \|_{\frac{2r}{r-1} ,\frac{2q}{q-1}}
\\
&\quad \quad 
+\|F^i,c\|_{r,q}\|z\|_{\frac{2r}{r-1} ,\frac{2q}{q-1}}^2 
+\|f^0,f^i\|_{2,2}(\|z\|_{2,2} + \|\nabla z\|_{2,2})
\\
&=:\omega _{\mathscr D}(s,t)\,,
\end{aligned}
\end{equation}
where for convenience, we shall from now on use the shorthand notation
\begin{equation}
\label{shorthand_notation}
\|\cdot \|_{a,b}:=\|\cdot \|_{L^a(s,t;L^b))}\,.
\end{equation}
Noticing that $\omega _{\mathscr D}(s,t)$ defined above is a control, one can apply Proposition \ref{pro:apriori} to obtain
\begin{equation}
\label{remainder_z2}
|z ^{2,\natural}_{st}|_{W^{-3,1}}
\leq C\left((t-s)^\alpha  \omega _{\mathscr D}(s,t)  + \|z \|^2_{L^\infty(s,t;L^2)}(t-s)^{3\alpha }\right).
\end{equation}
for every $(s,t)\in\Delta $ with $(t-s)\leq L$ for some absolute constant $L>0.$
We can now proceed to the proof of uniqueness.

\begin{trivlist}
	\item[\indent\textit{Energy inequality and application to uniqueness.}]
	One can now take $\phi =1\in W^{3,\infty}$ in \eqref{ito_square_free}. Putting the negative term on the left, using Assumption \ref{ass:A}, 
	and then making use of Young Inequality,
	we get for any $\epsilon >0$
	\begin{multline*}
	E_{st}
	:=
	(|z |_{L^2}^2)_{st} + \int_s^t|\nabla z_r|_{L^2}^2\dd r
	\\ 
	\lesssim_{\lambda}
	\left(\|F^i,b\|_{2r,2q}+ \|c\|_{r,q}\right)(\|z\|^2_{\frac{2r}{r-1} ,\frac{2q}{q-1}} + \|\nabla z\|^2_{2,2})
	+ \frac{1}{2\epsilon }  \|f^0,f^i\|_{2,2}^2 + \frac\epsilon 2(\|z\|^2_{2,2} + \|\nabla z\|_{2,2}^2)
	\\
	+ \int_{\T^d}\big(z ^2_s(B^{1,*}_{st}+B^{2,*}_{st})1+2z _s(Y^{-1}_{st}+\bb{Z^j}{\partial _jZ^{-1}}_{st})\big)\dd x  + \langle z ^{2,\natural}_{st},1\rangle\enskip.
	\end{multline*}
	Taking $\epsilon $ small enough and using the Interpolation Inequality gives then
	\begin{equation}
	\begin{aligned}
	E_{st}
	&\lesssim_{\lambda}
	\left(\|F^i,b\|_{2r,2q}+ \|c\|_{r,q}\right)\sup_{r\in[s,t]}E_r
	+|z _s|^2_{L^2}((t-s)^\alpha +(t-s)^{2\alpha })
	+|z ^{2,\natural}_{st}|_{W^{-3,1}},
	\end{aligned}
	\end{equation} 
	where we have estimated the remainder term by its $W^{-3,1}$-norm.
	
	Using the remainder estimate \eqref{remainder_z2}, we see that for $|t-s|\leq L(\lambda )$ (small enough):
	\[\begin{aligned}
	E_{st}
	\lesssim_{\lambda}
	\left(\|F^i,b\|_{2r,2q}+ \|c\|_{r,q} + (t-s)^\alpha\right)\sup_{r\in[s,t]}E_r + \|f^0,f^i\|_{2,2}^2\enskip.
	\end{aligned}
	\]
	By the Rough Gronwall estimate, Lemma \ref{lem:gronwall}, we infer the bound
	\begin{multline}
	\label{estimate:f}
	\|z \|^2_{L^\infty(0,T;L^2)}+\|\nabla z \|^2_{L^2(0,T;L^2)}
	\\
	\leq C\left(\lambda,\rho_\alpha(\Y),\|F^i,b\|_{L^{2r}(L^{2q})},\|c\|_{L^r(L^q)}\right)\Big(\exp \left\{\frac{T}{C _{\alpha,r,q,L}}\right\}\|z_0\|_{L^2}^2
	+\|f^0,f^i\|_{L^2(0,T;L^2)}^2
	\Big)
	.
	\end{multline}
	Next, note that the difference of two solutions of the same equation is itself a solution of a similar problem, with $f^0=f^i=0$ and $z_0=0.$
	Hence, the uniqueness follows from \eqref{estimate:f}.
\end{trivlist}

Existence relies on the same compactness argument as in the proof of Theorem \ref{thm:existence}, using a sequence
$\Y(n) \to \Y$ as in Assumption \ref{ass:geometric}, and the uniform estimate \eqref{estimate:f}.
We thus leave the details to the reader.

Similarly, the continuity statement is identical to that of \cite[Theorem 2.1]{hocquet2018ito}, and thus the proof is omitted.
\hfill\qed

%


\subsection{The main $L^\infty([0,T]\times\T^d)$ estimate}
We now make the assumption that $Y^{-1}=Y^0=0$ and prove a boundedness result for the following ansatz with (pure transport) rough input
\begin{align}
\label{parabolic_special}
\dd z_t +\left(-\mathscr M_tz_t + f^0 - \partial _if^i\right)\dd t = \dd \Y_t^i\partial _i z\,,\quad \text{on}\enskip [0,T)\times \T^d\,,
\\
z_0=0\,,
\nonumber
\intertext{where for the reader's convenience we recall that}
\mathscr M_tz=\partial _i\left(a^{ij}(t,x)\partial _jz +F^i(t,x)z\right) -b^i(t,x)\partial _iz - c(t,x)z\,.
\nonumber
\end{align}

Our main result is the following.

\begin{theorem}
	\label{thm:boundedness}
	Let $\Y,\mathscr M,a^{ij},\lambda ,F^i,b^i,c,f^0,f^i$ be as in Theorem \ref{thm:parabolic_gene}, but this time assume that $Y^{-1}=Y^0=0,$
	and that the inequality \eqref{subject_r_q} imposed on the pair $(r,q)$ is strict, namely
	\begin{equation}
	\label{conditions_strict}
	\frac{1}{r} + \frac{d}{2q} <1\,,
	\end{equation} 
	Suppose in addition that $f^0\in L^r(0,T;L^q),$ while $f^i\in L^{2r}(0,T;L^{2q}),i=1\dots d$, and let $z_0=0.$
	
	There is a constant
	\[
	C=C\left(T,\rho_\alpha (\Y),\lambda,\|F^i,b^i,f^i\|_{L^{2r}(L^{2q})},\|c,f^0\|_{L^r(L^q)},r,q\right)>0\,.
	\]
	which is non-increasing in $T>0$ and such that the following $L^\infty$-estimate holds:
	\begin{equation}
	\label{L_infty_rho}
	\|z \|_{L^\infty([0,T]\times\T^d)}\leq C\left(\|f^i,f^0\|_{L^1(0,T;L^1)}^{1/2}+\|z\|_{L^{\frac{2r}{r-1}}(0,T;L^{\frac{2q}{q-1}})}\right)\,.
	\end{equation}
\end{theorem}

Prior to addressing the proof of this result,
let us recall the following non-linear inequality, whose proof is standard and can be found in the literature, see e.g.\ \cite{ladyzhenskaya1968linear}.

\begin{lemma}
	\label{lem:recursion}
	Consider a sequence of positive numbers $U_n,n\geq 0,$ and
	constants $\beta >1$ and $C ,\tau >0$ 
	such that for all $n\geq 1:$
	\begin{equation}
	\label{hyp:recursive}
	U _{n+1} \leq C \tau ^nU_n^\beta .
	\end{equation} 
	Then, the following estimate holds: for any $n\geq 0$ we have
	\begin{equation}
	\label{estim:recursive}
	U_n\leq
	C^{\frac{\beta ^n-1}{\beta -1}} \tau ^{\frac{\beta ^n-1}{(\beta -1)^2}-\frac{n}{\beta -1}}U _0^{\beta ^n}.
	\end{equation} 
\end{lemma}

Following Moser's iteration pattern, we aim to apply Lemma \ref{lem:recursion}, for a well-chosen sequence $U_n$.

\begin{proof}[Proof of Theorem \ref{thm:boundedness}]
	Let $z $ be a solution, and assume without loss of generality that it is bounded (the general case is then deduced by approximation).
	For $\varkappa \geq 2,$ the product formula, Proposition \ref{pro:product}, asserts that for every $\phi \in W^{2,\infty}$:
	\begin{equation}
	\label{ito:real_power}
	|z |^\varkappa_{st} + \langle\mathscr D_{st}^{\varkappa},\phi \rangle = \int_{\T^d}|z _s|^\varkappa (B^{1,*} _{st} + B^{2,*}_{st})\phi \dd x
	+\langle z^{\varkappa,\natural}_{st},\phi \rangle
	\end{equation}
	where $z^{\varkappa,\natural}\in \CC(0,T;W^{-3,1}),$
	and similarly as before we denote by
	\[
	\begin{aligned}
	\langle\mathscr D_{st}^{\varkappa},\phi \rangle 
	&:= -\int_s^t\langle \varkappa z|z|^{\varkappa-2}\mathscr Mz,\phi \rangle\dd r
	\\
	&=\iint_{[s,t]\times \T^d} \Big[\varkappa (\varkappa-1)|z|^{\varkappa -2}\partial _iza^{ij}\partial _jz\phi  
	+ \varkappa z|z|^{\varkappa-2}\partial _jza^{ij}\partial _i\phi 
	\\
	&
	+ \varkappa(\varkappa-1)|z|^{\varkappa -2}F^iz\partial _iz \phi  + \varkappa |z|^{\varkappa}F^i\partial _i\phi 
	+ \varkappa z|z|^{\varkappa -2}b^i\partial _iz\phi + \varkappa c|z|^{\varkappa}\phi  
	\\
	&\quad 
	+\varkappa (\varkappa -1)|z|^{\varkappa -2}\partial _izf^i \phi
	+ \varkappa z|z|^{\varkappa -2}f^i\partial _i\phi 
	+\varkappa f^0z|z|^{\varkappa -2}\phi 
	\Big]\dd r\dd x\,.
	\end{aligned}
	\]
	
	Introducing the new unknown
	\[
	v:= |z |^{\varkappa/2},
	\]
	we see that the above expression simplifies as
	\[
	\begin{aligned}
	\langle\mathscr D_{st}^{\varkappa},\phi \rangle 
	&=\iint_{[s,t]\times \T^d} \Big[\frac{4(\varkappa-1)}{\varkappa}\partial _iva^{ij}\partial _jv\phi  
	+ 2v\partial _j va^{ij}\partial _i\phi 
	\\
	&\quad 
	+ 2(\varkappa-1)v\partial _ivF^i \phi  + \varkappa v^2F^i\partial _i\phi 
	+ 2b^iv\partial _iv\phi + \varkappa cv^2\phi  
	\\
	&\quad 
	+2(\varkappa -1)v\partial _ivf^i\phi
	+ \varkappa (\sgn z)v^{2\left(\frac{\varkappa -1}{\varkappa}\right)}f^i\partial _i\phi  
	+\varkappa (\sgn z)v^{2\left(\frac{\varkappa -1}{\varkappa}\right)}f^0\phi 
	\Big]\dd r\dd x\,,
	\end{aligned}
	\]
	where $\sgn(x)=1$ if $x>0$, $\sgn(x)=-1$ if $x<0$ and $\sgn(0)=0.$
	Using as before the notation $\|\cdot \|_{a ,b}:=\|\cdot \|_{L^a (s,t;L^b )},$
	then the bound
	$v^{2\left(\frac{\varkappa - 1}{\varkappa }\right)}\leq \frac{\varkappa-1}{\varkappa}v^2 + \frac1\varkappa$
	together with H\"older Inequality yield, uniformly in $\varkappa\geq 2$:
	\[
	\begin{aligned}
	|\mathscr D^{\varkappa}_{st}|_{W^{-1,1}}
	&\lesssim 
	\lambda^{-1}(\|\nabla v\|_{2,2}^2 +\|v\|_{\infty,2}(t-s)^{1/2}\|\nabla v\|_{2,2})
	\\
	&\quad \quad 
	+\varkappa \left(\|F\|_{2r,2q}\|v\|_{\frac{2r}{r-1},\frac{2q}{q-1}}\|\nabla v\|_{2,2}+\|F\|_{r,q}\|v\|^2_{\frac{2r}{r-1},\frac{2q}{q-1}}\right)
	\\
	&\quad \quad 
	+\|b\|_{2r,2q}\|v\|_{\frac{2r}{r-1},\frac{2q}{q-1}}\|\nabla v\|_{2,2}
	+\varkappa \|c\|_{r,q}\|v\|_{\frac{2r}{r-1},\frac{2q}{q-1}}^2
	\\
	&\quad \quad 
	+ \varkappa \left(\|f^i\|_{2r,2q}\|v\|_{\frac{2r}{r-1},\frac{2q}{q-1}}\|\nabla v\|_{2,2} + \|f^i,f^0\|_{r,q}\|v\|^2_{\frac{2r}{r-1},\frac{2q}{q-1}} +\|f^i,f^0\|_{1,1}\right)\,.
	\end{aligned}
	\]
	Furthermore, Young's inequality shows that for any $\epsilon >0$:
	\begin{equation}
	\begin{aligned}
	&|\mathscr D^{\varkappa}_{st}|_{W^{-1,1}}
	\lesssim
	\left[\|v\|_{\infty,2}^2(t-s)
	+\|\nabla v\|^2_{2,2} \left(\lambda^{-1} + 1 + \frac32\epsilon \right)\right]
	\\
	&
	+\varkappa ^2\|v\|_{\frac{2r}{r-1},\frac{2q}{q-1}}^2
	\Bigg[
	\|F\|_{r,q} + \|c\|_{r,q} +\|f^i,f^0\|_{r,q}
	+\frac{1}{2\epsilon }\big(\|F\|^2_{2r,2q} +\|b\|_{2r,2q} 
	+ \|f^i\|_{2r,2q}^2\big)\Bigg]
	\\
	&\quad \quad \quad 
	+ \varkappa \|f^i,f^0\|_{1,1}
	\\
	&\quad \quad 
	=:\omega ^1_\varkappa(s,t) + \omega ^2_\varkappa(s,t) + \omega _\varkappa ^3(s,t)\,.
	\end{aligned}
	\end{equation} 
	The terms in the right hand side of the above inequality have complementary Young integrability, hence it follows that $\omega ^i_\varkappa,i=1,2,3,$ are  control functions.
	Therefore, we can proceed as in the proof of Theorem \ref{thm:existence}.
	Testing equation \eqref{ito:real_power} against $\phi =1$, and then estimating the remainder thanks to Proposition \ref{pro:apriori}, we have for $|t-s|\leq L(\alpha ,\rho_\alpha(\Y) )$ small enough:
	\[(|v|^2_{L^2})_{st}
	+ \int_s^t|\nabla v|^2_{L^2}\dd r
	\lesssim _{\lambda ,\rho _\alpha (\Y)}
	\|v\|^2_{L^\infty(s,t;L^2)}(t-s)^\alpha
	+ (t-s)^\alpha \sum_{i=1}^3\omega^i _\varkappa (s,t)\,.
	\]
	Taking $L(\alpha ,\rho_\alpha (\Y),\lambda )>0$ smaller if necessary, the first control $\omega ^1_\varkappa(s,t)$ can be absorbed to the left.
	Introducing the exponents
	\[
	\varrho_0:= \frac{2r}{r-1}\,,\quad 
	\sigma _0 = \frac{2q}{q-1}\,\,,
	\]
	we thus obtain the inequality
	\begin{equation}
	\begin{aligned}
	\label{estim:v_B}
	\sup_{t\in [0,T]}|v_t|^2_{L^2}
	&+ \int_0^T|\nabla v_t|^2_{L^2}\dd t
	\\
	&\leq C\left(\rho_\alpha(\Y),\lambda ,\|F^i,b^i,f^i\|_{2r,2q},\|c,f^0\|_{r,q}\right)\varkappa^2\left(\|v\|_{\varrho _0,\sigma _0}^2+\|f^i,f^0\|_{1,1}\right).
	\end{aligned}
	\end{equation}
	
	We now aim to apply Lemma \ref{lem:recursion} for a well-chosen sequence $U_n,n\geq 0.$
	To this end, take any
	\begin{equation}
	\label{choice_eps}
	0<\epsilon<\frac{2}{d}\left (1-\frac{1}{r} - \frac{d}{2q}\right )\,.
	\end{equation}
	and define 
	\[
	\beta :=1+\epsilon \,.
	\]
	Observe by \eqref{choice_eps} that 
	\[
	\frac{1}{\varrho_0\beta}+ \frac{d}{2\beta \sigma _0}\geq \frac{d}{4},
	\]
	which means in particular that the increased exponents 
	\[
	\varrho _1:=\varrho _0\beta ,\quad 
	\sigma _1:=\sigma _0\beta \,,
	\]
	still satisfy the condition \eqref{conditions}. Therefore, 
	letting $\varkappa=\varkappa_0:= 2$,
	then \eqref{interpolation_inequality}, \eqref{estim:v_B} and the basic inequality $\sqrt{a+b}\leq \sqrt{a}+\sqrt{b}$ give
	\begin{equation}\label{est:U0}
	U_0:=\|z\|_{\varrho_1,\sigma _1}
	\leq \widetilde C 
	\left(\|f^i,f^0\|_{1,1}^{1/2}+\|z\|_{\varrho _0,\sigma _0}\right)
	\end{equation}
	for another such constant $\widetilde C>0,$ independent of $\beta $.
	Next, for $n\geq 1$ define
	\[
	\varkappa_n:= 2\beta ^n\,,\quad 
	\quad
	\varrho _n:=\varrho _0\beta ^n,\quad 
	\quad
	\sigma _n:=\sigma _0\beta ^n\,,\quad 
	\quad
	U_n:=\|z \|^{\beta ^n}_{\varrho_{n+1},\sigma_{n+1}}+1\,.
	\]
	By \eqref{estim:v_B} and our choice of $\epsilon =\beta -1$ in \eqref{choice_eps}, we obtain similarly for every $n\geq 0$
	\[\begin{aligned}
	\|z\|^{\beta ^n}_{\varrho_{n+1},\sigma_{n+1}}=\||z|^{\beta^n}\|_{\varrho_1,\sigma _1} 
	&\leq \widetilde C\beta ^n\left(\|f^i,f^0\|_{1,1}+\||z|^{\beta ^n}\|_{\varrho_0,\sigma_0} \right)
	\\
	&\leq \hat C\beta ^n\left(1+\|z\|^{\beta ^{n-1}}_{\varrho_n,\sigma_n}\right)^\beta
	\end{aligned}
	\]
	where for convenience we now put the dependency on $\|f^i,f^0\|_{1,1}$ inside the constant $\hat C.$
	Thus, we obtain the recursive inequality
	\[
	U_{n+1}
	\leq  \hat C \beta ^nU_{n}^{\beta }\,,\quad \text{for all}\enskip n\geq 0\,,
	\]
	which by Lemma \ref{lem:recursion} implies that
	\begin{equation}
	U_n\leq 
	\hat C ^{\frac{\beta ^n-1}{\epsilon }}\beta ^{\frac{\beta ^n-1}{\epsilon ^2}-\frac{n}{\epsilon }}U_0^{\beta ^n}\,\,.
	\end{equation}
	Taking the $\beta ^n$-th root, letting $n\to \infty$ and then using \eqref{est:U0} to bound the initial term, we infer the estimate
	\begin{equation}
	\|z \|_{L^\infty([0,T]\times\T^d)}\equiv\lim_{n\to\infty}(U_n)^{\beta ^{-n}}
	\leq C\left(\|f^i,f^0\|_{L^1(0,T;L^1)}^{1/2}+\|z\|_{L^{\frac{2r}{r-1}}(0,T;L^{\frac{2q}{q-1}})}\right)
	\end{equation}
	for some new constant $C>0,$ hence proving \eqref{L_infty_rho} when $z$ is bounded.
	
	The general case follows by approximation, using the geometricity of $\Y$ and the fact that the constant $C$ in the above inequality only depends on $\rho _\alpha (\Y).$ Since this is similar to \cite{hocquet2018ito}, we omit the details.
\end{proof}

\subsection{Backward RPDEs}
\label{subsec:backward}
In this subsection, we investigate backward-type RPDEs of the form
\begin{equation}
\label{ansatz_backward}\left\{
\begin{aligned}
&\dd \sigma + (\partial _if^i - f^0)\dd t = \dd \mathbf P(g) \quad \text{on}\enskip [0,T]\times\T^d
\\
&\sigma _T(\cdot )=\sigma ^T\in L^p(\T^d)
\end{aligned}\right.
\end{equation} 
with unknown $\sigma_t(x),$ and show its equivalence with a solution of a forward problem of the same form.

Here, what plays the role of the unbounded rough driver $\Q$ in earlier sections is a pair $\mathbf P=(P^1,P^2),$ of $2$-index families of affine linear differential operators. In contrast with Definition \ref{def:Q} however, we impose the \textit{backward} Chen's relations (this terminology is not standard)
\begin{equation}
\label{backward_chen}
\left[\begin{aligned}
&\delta P^1_{s\theta t }=0
\\
&\delta P^2_{s\theta t}= -\widetilde P^1_{s\theta }\circ P^1_{\theta t}
\end{aligned}\right.
\end{equation}
for all $(s,\theta,t)\in\Delta_2$,
where $\widetilde P^1_{st}:=(P^{1}-P^1(0))_{st}.$

\begin{example}
	\label{exa:B_star}
	Let $\B=(B^1,B^2)$ be as in Example \ref{example:URD}. Define the two-parameter family of differential operators $\mathbf P=(P^1,P^2)$ as
	\[ 
	P^1_{st}:= -(B^1_{st})^*,\quad 
	P^2_{st}:= -(B^2_{st})^*\quad (s,t)\in\Delta.
	\]
	Then, for any $(s,\theta,t)\in \Delta_2$ we have
	$\delta P^1_{s \theta t}= -(\delta B^1_{s\theta t})^*=0$
	while
	\[ 
	\delta P^2_{s\theta t}= -(B^1_{\theta t}\circ B^1_{s\theta})^*
	= - (B^1_{s\theta})^* \circ (B^1_{\theta t})^*
	= - P^1_{s\theta}\circ P^1_{\theta t}\enskip ,
	\]
	hence $\mathbf P$ satisfies the backward chen's relations \eqref{backward_chen}.
\end{example}

First, we need a definition.
\begin{definition}
	\label{def:backward_solution}
	Let $T>0$, $\alpha \in(1/3,1/2]$ and fix $p\in [1,\infty].$
	Assume that we are given
	$f^i\in L^1(0,T;L^p),$ $i=0,\dots d,$ and that $g$ is controlled by $P$ with backward Gubinelli derivative $g',$
	in the sense that there is a control function $\varpi $ and a path $g'\in L^\infty(0,T;L^p)$ such that the two-parameter element
	\[
	\mathscr R_{st}^{g}:= g_{st} - P_{st}^1(g'_t)\quad (s,t)\in\Delta \,,
	\]
	verifies the bound 
	\[
	|\mathscr R_{st}|_{W^{-2,p}}\leq \varpi (s,t)\,,\quad \forall (s,t)\in\Delta \,.
	\]
	
	\begin{itemize}	
		\item
		A mapping $\sigma \colon[0,T]\to L^p$ is called a {\emph{backward $L^p$-weak solution}} to the
		rough PDE \eqref{rough_PDE_gene} if it fulfills the following conditions
		\begin{enumerate}[label=(\arabic*)]
			\item $\sigma \colon[0,T]\to L^p$ is weakly-$*$ continuous and belongs to $L^\infty(0,T;L^p)$;
			\item for every $\phi \in W^{2,p'}$ with $1/p+1/p'=1$, and every $(s,t)\in\Delta :$
			\begin{equation}
			\label{nota:solution_backward}
			\int_{\T^d}\sigma _{st}\phi \dd x=\iint_{[s,t]\times\T^d} (f^i_r\partial _i\phi - f^0_r\phi ) \dd x\dd r
			+ \Big\langle P^1_{st}(g_t)+P^2_{st}(g'_t)+\sigma _{st}^\natural,\phi \Big\rangle\,,
			\end{equation}
			for some $\sigma ^\natural\in \mathcal C^{1+}(0,T;W^{-3,p}).$
		\end{enumerate}
		
		\item
		As before, we call it a \emph{backward $L^p$-energy solution} of \eqref{rough_PDE_gene} if in addition
		$v$ belongs to $L^p(0,T; W^{1,p}).$
	\end{itemize}
\end{definition}

\begin{proposition}
	\label{pro:equivalence}
	Let $f,g,g',\Y$ be as in Definition \ref{def:solution}, and for $0\leq s\leq t\leq T,$ introduce 
	\[
	\check f_{t}:=f_{T-t}\,,
	\quad 
	\check g_{t}:=g_{T-t}\,,
	\quad 
	\check g'_{t}:=g'_{T-t}\,,
	\]
	and for $i=1,2,$ denote by
	\[
	\check P_{st}^i:=-P_{T-t,T-s}^i\,.
	\]
	
	Then, $\mathbf P$ fulfills the backward Chen's relations \eqref{backward_chen} if an only if $\mathbf{\check P}$ satisfies the forward analogue, that is \eqref{affine_chen}. Moreover, for such $\mathbf P,$ then $\sigma  _t$ is a solution of the backward RPDE \eqref{ansatz_backward}
	if and only if $v _t:=\sigma_{T-t}$ is a solution of the forward equation
	\[\left\{
	\begin{aligned}
	&\dd v_t +(-\partial _i\check f^i_t + \check f^0_t)\dd t = \dd \mathbf{\check P}_t(v_t),\quad \text{on}\enskip [0,T]\times \T^d\,,
	\\
	&v_0:=\sigma ^T\,.
	\end{aligned}\right.
	\]
\end{proposition}

\begin{proof}
	We only prove the ``only if'' part, the opposite direction being similar.
	First, observe that thanks to \eqref{backward_chen}, we have
	\[
	\delta \check P^1_{s\theta t}=0
	\]
	and
	\[\begin{aligned}
	\delta \check P^2_{s\theta t}
	&= -P^2_{T-t,T-s} + P^2_{T-t,T-\theta } + P^2_{T-\theta ,T-s}
	\\
	&= \widetilde P^1_{T-t,T-\theta }\circ P^1_{T-\theta ,T-s}
	= \widetilde{(\check P^1_{\theta t})}\circ\check P^1_{s\theta }\,,
	\end{aligned}
	\]
	showing that $\mathbf{\check P}$ satisfies \eqref{affine_chen}.

	Next, let $0\leq s\leq t \leq T$ and denote by $t':=T-t$ while $s':=T-s.$ 
	Assuming that $\sigma _t$ is a solution of the backward problem, since $t'\leq s'$ we have by definition of a backward solution:
	\[
	\begin{aligned}
	&\int_{\T^d}(v_{st})\phi \dd x
	=-\int_{\T^d} (\sigma _{s'}-\sigma _{t'})\phi \dd x
	\\
	&=-\iint_{[t',s']\times\T^d} (-f^i_r\partial _i\phi + f^0_r\phi ) \dd x\dd r
	+ \Big\langle -P^1_{t's'}(g_{s'})-P^2_{t's'}(g'_{s'})+\sigma _{t's'}^\natural,\phi \Big\rangle\,,
	\\
	&=\iint_{[s,t]\times\T^d} (-\check f^i_{\tau }\partial _i\phi + \check f^0_{\tau }\phi ) \dd x\dd \tau + 
	\Big\langle \check P^1_{st}(\check g'_{t})+\check P^2_{st}(\check g'_t)+\sigma _{T-t,T-s}^\natural,\phi \Big\rangle
	\end{aligned}
	\]
	Let $v^\natural_{st}:=\sigma _{T-t,T-s}^\natural, $ and take $z>1$ and a control $\omega \colon\Delta \to\R_+$ so that $|\sigma ^{\natural}_{t's'}|_{W^{-3,p}}\leq \omega (t',s')^z.$
	To conclude, it is sufficient to observe that $\check \omega (s,t):=\omega (T-t,T-s), $ $0\leq s\leq t\leq T$ is a control function. But this is an immediate consequence of the corresponding property for $\omega $.
\end{proof}

As a consequence of this equivalence property and our forward solvability results, Theorem \ref{thm:parabolic_gene}, we easily obtain the following statement.
\begin{corollary}
	\label{cor:backward}
	Fix any $\sigma ^T\in L^2,$ and let $\mathscr M$ be as in \eqref{general_M}. 
	Consider a family of differential operators $\mathbf P=(P^1_{st},P^2_{st})$ satisfying \eqref{backward_chen}.
	Assume that $\Q:=\mathbf{\check P}$ is geometric, in the sense that there is some triad of coefficients $(Y;\bb{Y^j}{\partial_j Y};\bb{Y^j}{\partial_j Y^{-1}}+\bb{Y^0}{Y^{-1}})$ satisfying Assumption \ref{ass:solution_coef}, and such that $\Q$ is given by the relation \eqref{Q_Y}.
	
	Then, there exists a unique solution $\sigma $ to the backward problem
	\[
	\left\{
	\begin{aligned}
	&\dd \sigma  + (\mathscr A\sigma +f)\dd t = \dd \mathbf P^i(\sigma)
	\quad \text{on}\enskip [0,T]\times \T^d\,,
	\\
	& \sigma _T=\sigma ^T\,.
	\end{aligned}\right.
	\]
\end{corollary}

\begin{proof}
	This is immediate by the use of Proposition \ref{pro:equivalence}, Theorem \ref{thm:parabolic_gene}, and the involution $(v_t)_{t\in[0,T]}\mapsto \check v\equiv(v_{T-t})_{t\in[0,T]}$. Details are left to the reader.
\end{proof}

In our way to prove the second uniqueness result, Theorem \ref{thm:regular}, we will need a rigorous justification of the weighted inequality \eqref{product_v_m}.
For that purpose, we now state an extension of the product formula (Proposition \ref{pro:product}) that applies for backward versus forward solutions. For simplicity, we only consider the rough input given by Example \ref{example:URD}, but several generalizations should be possible. 

\begin{proposition}[Forward versus backward product formula]
	\label{pro:product_backward}
	Let $\B=\B\{\X\}$ be the unbounded rough driver given by Example \ref{example:URD}, where $\X$ satisfies Assumption \ref{ass:geometric}.
	
	Fix $p,p'\in [1,\infty]$ with $1/p+1/p'=1,$ and consider $u\colon[0,T]\to L^p,$ $v\to L^{p'}$ such that
	\[
	\left\{\begin{aligned}
	&u\in \mathcal D_{B}^{\alpha ,p}
	\\
	&\dd u = F\dd t + \dd \B u\,,
	\end{aligned}\right.
	\]
	in $L^p$ (forward sense),
	while $v$ solves the following backward problem, in $L^{p'}$:
	\[
	\left\{\begin{aligned}
	&v\in \mathcal D_{-B^*}^{\alpha ,p'}
	\\
	&\dd v = G\dd t  -\dd \B ^* v 
	\end{aligned}\right.
	\]
	(in the sense of Definition \ref{def:backward_solution} with $\Q$ as in Example \eqref{exa:B_star}).
	Assume furthermore that $u(\cdot )G(\cdot -a) $ and $ F(\cdot -a) v(\cdot )$ are well-defined as distributions in $L^1(0,T;W^{-1,1}),$ for any $a\in \R^d$ with $|a|\leq 1.$ 
	
	Then, the pointwise product $uv$ has finite variation as a path with values in $W^{-1,1},$
	and it is a forward, $L^1$-solution of the following partial differential equation
	\begin{equation}
	\dd (uv)= \big[uG +Fv\big]\dd t\,\quad \text{on}\enskip [0,T]\times \T^d\,.
	\end{equation}
\end{proposition}

\begin{proof}
	By definition of the backward equation,
	we have $v^{\natural,-}\in \CC(W^{-3,p'}),$ where for each $s\leq t$:
	\[\begin{aligned}
	v^{\natural,-}_{st}
	&:=v_{st} - \int_s^tG\dd r +B_{st}^{1,*}v_t +B_{st}^{2,*}v_t 
	\\
	&=v_{st} - \int_s^tG\dd r +B_{st}^{1,*}v_s +B_{st}^{2,*}v_s + (B_{st}^{1,*} +B_{st}^{2,*})v_{st}\,.
	\end{aligned}
	\]
	Letting 
	\[
	v^{\natural,+}_{st}:=v_{st} + B^{1,*}_{st}v_s + (B^{2,*}_{st} - (B^{1,*}_{st})^2)v_s
	\]
	it is easily observed that $v^{\natural,+}_{st}$ also belongs to $\CC(W^{-3,1}).$
	Defining
	\[(u\otimes v)_t(x,y):=u_t(x)v_t(y)\,,
	\]
	and then using the identity
	$(u\otimes v)_{st} = u_{st}\otimes v_s + u_{st}\otimes v_{st} + u_s\otimes v_{st}$,
	it is seen by immediate algebraic computations that
	\begin{equation}
	\label{u_tensor_v}
	\begin{aligned}
	(u\otimes v)_{st}
	&- \big[B^1_{st}\otimes 1-1\otimes B^{1,*}_{st}\big](u\otimes v)_s
	\\
	&-\big[B^2_{st}\otimes1 -B^1_{st}\otimes B^{1,*}_{st} -1\otimes (B^{2,*}_{st} - (B^{1,*}_{st})^2)\big](u\otimes v)_s
	\\
	&= u^{\natural}_{st}\otimes v_s + u_s \otimes v^{\natural,-}_{st} 
	+R^u_{st}\otimes v_{st} + B^1_{st}u_s\otimes \mathscr R^v_{st} - B^1_{st}\otimes B^{1,*}_{st}v_{st}\,.
	\end{aligned}
	\end{equation} 
	where we recall the notations
	\[
	R^u_{st}=u_{st}-B_{st}u_s,\quad \mathscr R^v_{st}=v_{st}+B^*_{st}v_t\,.
	\]
	In the formula \eqref{u_tensor_v}, it is easy to see that each term lies in $C^{3\alpha }_2(W^{-3,1})$; the proof of this fact follows exactly the lines of \cite[Section 5]{hocquet2017energy}, and therefore we omit the details.
	Next, define $\boldsymbol{\Gamma} _{st}=(\Gamma ^1,\Gamma ^2)$ by
	\[\left[\begin{aligned}
	&\Gamma ^1_{st}:= B^1_{st}\otimes 1 -1 \otimes B_{st}^{1,*}
	\\
	&\Gamma ^2_{st}:= B^2_{st}\otimes 1 -B^1_{st} \otimes B_{st}^{1,*} + 1\otimes ((B^{1,*}_{st})^2 -B^{2,*}_{st})
	\end{aligned}\right.
	\]
	Using the forward and backward Chen's relations, we have $\delta \Gamma ^1_{s\theta t}=0$ and
	\[\begin{aligned}
	\delta \Gamma ^2_{s\theta t}
	&= \delta B^2_{s\theta t}\otimes 1 - \left(B^1_{\theta t}\otimes B^{1,*}_{s\theta } + B^1_{s\theta }\otimes B^{1,*}_{\theta t}\right) 
	\\
	&\quad \quad \quad \quad \quad 
	+ 1\otimes \left(B^{1,*}_{s\theta }\circ B^{1,*}_{\theta t}+ B^{1,*}_{\theta t}\circ B^{1,*}_{s\theta } - \delta B^{2,*}_{s\theta t}\right)
	\\
	&=\Gamma ^1_{\theta t}\circ \Gamma ^1_{s\theta }\,,
	\end{aligned}
	\]
	thus $\boldsymbol\Gamma $ fulfills the usual Chen's relations for unbounded rough drivers.
	
	In particular, we infer from \cite[Section 5]{hocquet2017energy} that $u\otimes v$ is a forward solution to the problem
	\[
	\dd (u\otimes v) = (u\otimes G+ F\otimes v)\dd t + \dd \boldsymbol\Gamma (u\otimes v)\,,
	\]
	in the sense of the corresponding Euler-Taylor expansion \eqref{euler-taylor_Q}, in $W^{-3,1}(\T^d\times\T^d).$
	
	Now, the remainder of the proof is carried out by the exact same steps as in \cite[Section 6]{hocquet2017energy}, hence we leave the details to the reader.
\end{proof}

\section{A rough transport problem}
\label{sec:transport}
In this section, we consider yet another $(d+2)$-dimensional path $Z^i\in C^\alpha(0,T;W^{3,\infty})$ for $i=-1,0,\dots,d$ where $\alpha>1/3.$
Our aim is solve the following rough transport problem in $L^\infty(\mathbb T^d)$:
find a controlled path $\Phi\colon[0,T]\to L^\infty(\T^d)$, solution to the (forward) rough transport equation
\begin{equation}
\label{transport_gene}
\left\{\begin{aligned}
&\dd  \Phi = \dd \Z^i\partial _i\Phi + \dd \Z^{-1},\quad [0,T]\times\T^d\,,
\\
&\Phi _0=0\,,
\end{aligned}\right.
\end{equation} 
with an enhancement $\Z$ of $Z$ as in Assumption \ref{ass:solution_coef}.
For simplicity only, we consider no zero order multiplicative term $Z^0$ in \eqref{transport_gene} (this in turn will be enough for the purpose of proving Theorem \ref{thm:regular}). For notational coherence we take a similar labelling as in Section \ref{sec:parabolic} by letting $Z^0=0$ and denoting the ``source term'' by the $-1$ exponent. By geometricity, it is immediate to see that the entries $Z^0,\bb{Z^0}{Z^{-1}}$ and $\bb{Z^{\mu}}{\partial _\mu Z^0}$ vanish identically. 

Our main result in this section is the following.

\begin{theorem}
	\label{thm:transport}
	Let $\Z $ be as in Assumption \ref{ass:solution_coef} such that $Z^0=\bb{Z^0}{Z^{-1}}=\bb{Z^{\mu}}{\partial _\mu Z^0}=0.$
	The following holds true.
	
	\begin{enumerate}[label=(\roman*)]
		\item \label{transport_i}
		There exists $\Phi\colon [0,T]\to \cap_{p\in[1,\infty)}L^p$ such that for every $p\in[1,\infty)$, $\Phi$ is an $L^p$-weak solution of \eqref{transport_gene}, in the sense of Definition \ref{def:solution}. It admits moments of arbitrary order, more precisely: for each $n$ there exists a finite constant $C(n)>0$ such that 
		\[
		\|\Phi\|_{L^\infty(L^n)}\leq C (n)\,.
		\]
		
		\item \label{transport_ii} Letting 
		$\mathcal E_t(x):=\exp -\Phi_t(x) $, we have
		\[
		\dd \mathcal E=(\dd \Z^i\partial _i\mathcal  - \dd \Z^{-1})\mathcal E\,,
		\]
		in the sense of the following Euler-Taylor expansion in $W^{-2,1}$
		\[\begin{aligned}
		&\mathcal E_{st}
		=\left(Z_{st}^i\partial _i  -Z^{-1}_{st}\right)\mathcal E_s
		\\
		&
		+ \Big(\frac12Z^i_{st}Z^j_{st}\partial _{ij} + \left[\bb{Z^\mu }{\partial _\mu Z}^i_{st}-Z^{-1}_{st}Z^i_{st}\right]\partial _i -\bb{Z^\mu }{\partial _\mu Z^{-1}}_{st} + \frac12(Z^{-1}_{st})^2\Big)\mathcal E_s +\mathcal E_{st}^\natural\,,
		\end{aligned}
		\]
		for $(s,t)\in\Delta$
		and where $\mathcal E^\natural\in\CC(0,T;W^{-3,1}).$
		
		\item \label{transport_iii}
		Set $d=1$, and assume that $\rho_\alpha(\partial_x \Z)<\infty.$
		Then, the solution obtained above belongs to $L^\infty([0,T]\times \mathbb T^1),$ and its spatial derivative $\partial _x\Phi$ also has moments of arbitary order.
	\end{enumerate}
\end{theorem}

\begin{proof}
	Assume first that $Z$ is in $C^1(0,T;W^{3,\infty})$. By consistency with Stieltjes integration, there exists a distributional solution $\Phi \in L^1([0,T]\times \mathbb T^d)$, and moeover, it follows by standard results that $\Phi$ is essentially bounded (in particular it lies in $L^\infty(L^p)$ for all $p\geq1$).
	We now aim to obtain a priori estimates for $\Phi $ in $L^\infty$, which only depend on $\rho _\alpha (\Z).$
	
	To wit, note that by \eqref{transport_gene} and the Sewing Lemma, there is a unique remainder term $\Phi ^\natural$ so that
	\[\Phi_{st} = Z^i_{st}\partial _i\Phi_s + Z_{st}^{-1} + \left(\frac12Z^i_{st}Z^j_{st}\partial _{ij}+ \bb{Z^\mu }{\partial _\mu Z}_{st}^i\partial _i\right)\Phi _s + \bb{Z^\mu }{\partial _\mu Z}^{-1}_{st} + \Phi ^{\natural}_{st}\]
	and the generalized Chen's relations \eqref{gene_chen} give for all $0\leq s\leq \theta\leq t\leq T$
	\[
	\delta \Phi ^\natural_{s\theta t}= Z^i_{\theta t}\partial _i\left(\Phi _{s\theta }- Z^j_{s\theta }\partial _j\Phi _s- Z_{s\theta }^{-1}\right)
	+ \left(\frac12Z^i_{\theta t}Z^j_{\theta t}\partial _{ij} + \bb{Z^\mu }{\partial _\mu Z}^i_{\theta t}\partial _i\right)\Phi _{s\theta }
	\]
	Applying now the remainder estimates, Proposition \ref{pro:apriori}, we obtain 
	\[
	|\Phi ^{\natural}_{st}|_{W^{-3,2}} \lesssim \rho _\alpha (\Z)(t-s)^{3\alpha }\,,
	\]
	and consequently:
	\[
	\|\Phi \|_{\mathcal D_{Q\{Z\}}^\alpha }\leq C\rho _\alpha (\Z)\,,
	\]
	where we let $\Q\{\Z\}$ be the unbounded rough driver \eqref{Q_Y} with coefficient path $\Z$, and we recall that $\mathcal D_{Q\{Z\}}$ is the controlled paths space associated with the latter.

	Applying Proposition \ref{pro:product} recursively, an easy induction shows that for each $n\geq 2,$ $\Phi ^n$ satisfies
	\begin{equation}
	\label{eq:Phi_n}
	\begin{aligned}
	&(\Phi^n)_{st} 
	= \left(Z^i_{st}\partial _i +\frac12Z^i_{st}Z^j_{st}\partial _{ij}+ \bb{Z^\mu }{\partial _\mu Z}_{st}^i\partial _i\right)(\Phi _s^n) 
	\\
	&\quad \quad 
	+ n\left(Z_{st}^{-1} + \bb{Z^\mu }{\partial _\mu Z}^{-1}_{st} + Z^{-1}_{st}Z^i_{st}\partial _i \right)\left(\Phi _s^{n-1}\right) + \frac{n(n-1)}{2}(Z^{-1}_{st})^2\Phi _s^{n-2}
	+ \Phi ^{n,\natural}_{st}
	\end{aligned}
	\end{equation}
	Again, applying $\delta$ to both sides, we have thanks to \eqref{gene_chen}
	\begin{multline}
	\label{delta_Phi}
	\delta \Phi ^{n,\natural}_{s\theta t}
	= Z^i_{\theta t}\partial _i\left(\Phi ^n_{s\theta }-Z^j_{s\theta }\partial _j\Phi ^n_s - nZ^{-1}_{s\theta }\Phi _s^{n-1}\right) + \left(\bb{Z^\mu }{\partial _\mu Z}^i_{\theta t} + \frac12Z^i_{\theta t}Z^j_{\theta t}\partial _{ij}\right)\Phi ^n_{s\theta }
	\\
	+ nZ^{-1}_{\theta t}\left(\Phi _{s\theta }^{n-1}- Z^i_{s\theta }\partial _i\Phi _s^{n-1}-(n-1)Z^{-1}_{s\theta }\Phi ^{n-2}_s\right)
	+ n \left(\bb{Z^\mu }{\partial _\mu Z}^{-1}_{\theta t}+Z^{-1}_{\theta t}Z^i_{\theta t}\partial _i\right)\Phi _{s\theta }^{n-1}
	\\
	+\frac{n(n-1)}{2}(Z^{-1}_{\theta t})^2\Phi^{n-2}_{s\theta}\,.
	\end{multline}
	By the continuity part of the Sewing Lemma, we obtain that
	\[
	\|\Phi ^{n,\natural}\|_{\mathcal D^\alpha _{Q^Z}}\lesssim n\rho _\alpha (\Z)\|\Phi ^{n-1}\|_{\mathcal D ^\alpha_{Q^Z} }\,,
	\]
	and therefore, by induction
	\[
	\|\Phi ^{n,\natural}\|_{W^{-3,1}}\lesssim n!\rho _\alpha (\Z)^n \|\Phi ^\natural\|_{W^{-3,1}}\lesssim n!\rho _\alpha (\Z)^n \|\Phi ^\natural\|_{W^{-3,2}}
	\leq C(n,\rho _\alpha (\Z))\,.
	\]
	
	We can now let $n=2m$  and test \eqref{eq:Phi_n} against $\phi =1$, the remainder $\langle \Phi ^{n,\natural},1\rangle$ being estimated by its $W^{-3,1}$-norm thanks to the previous computations.
	But then, the rough Gronwall Lemma asserts the existence of a constant $C_1\left (n,\rho _\alpha (\Z)\right )>0$
	such that
	\begin{equation}\label{est:Phi_n}
	\|\Phi \|_{L^\infty(L^n)}\leq \widetilde C_1\left (n,\rho _\alpha (\Z)\right )\,.
	\end{equation}
	This finishes the proof of \ref{transport_i}.\smallskip

	In order to show \ref{transport_iii},
	we are now going to show an a priori estimate derived from the equation satisfied by $\Psi _t(x):=\partial _x\Phi_t(x)$.
	The proof of is essentially similar to the previous one, but requires to introduce some notation.
	
	Denoting by $\Psi^{\natural}:=\partial_x\Phi^{\natural}$ and considering test functions of the form $-\partial _x\phi$ in \eqref{eq:Phi_n}, we infer the following Euler-Taylor expansion on $\Psi$:
	\begin{equation}
	\begin{aligned}
	\label{eq:Psi}
	&\Psi _{st}-\Psi^{\natural}_{st}
	\\
	&= Z_{st}^1\partial _i\Psi  + \partial_x Z^1_{st}\Psi + \partial _xZ^{-1}_{st}
	\\
	&\quad 
	+ \Big(\frac12(Z^1_{st})^2\partial _{xx} +\bb{Z^1}{\partial _x Z^1}_{st}\partial _x
	+(\partial _xZ^1_{st})Z^1_{st}\partial _x 
	+ \partial _x \bb{Z^1}{\partial _x Z^1}_{st}\Big)\Psi _s 
	\\
	&\quad \quad 
	+ \partial_x\bb{Z^1}{\partial_x Z^{-1}}_{st}  
	\\
	&=Y^1_{st}\partial _i\Psi  + Y^0_{st}\Psi + Y^{-1}_{st}
	\\
	&
	+ \Big(\frac12(Y^1_{st})^2\partial _{xx} + Y^0_{st}Y^1_{st}\partial _x +\bb{Y^1}{\partial_x Y^1}_{st}\partial _x +  \partial _x \bb{Z^1}{\partial _x Z^1}_{st}\Big)\Psi _s 
	+ \partial_x\bb{Z^1}{\partial_x Z^{-1}}_{st}  
	\end{aligned}
	\end{equation}
	where we introduce $Y^1:=Z^1,$ $Y^0:=\partial_x Z^1$, $Y^{-1}:= \partial_x Z^{-1}.$ Introducing the enhancement
	\begin{equation}\label{odd_enhancement}
	\begin{aligned}
	&\Y
	=\Big(\left[Y^{-1},Y^0,Y^1\right ];\left [\bb{Y^1}{\partial_x Y^{0}},\bb{Y^1}{\partial_x Y^{1}}\right]; \bb{Y^1}{\partial_x Y^{-1}}+\bb{Y^0}{Y^{-1}} \Big)
	\\
	&:=\Bigg(\left [\partial_x Z^{-1},\partial_x Z^1,Z^1\right ];
	\,
	\left[\partial_x\bb{Z^1}{\partial_x Z^1} - \frac12(\partial_x Z^1)^2
	,\bb{Z^1}{\partial_x Z^1}\right];
	\, 
	\partial_x\bb{Z^1}{\partial_x Z^{-1}}\Bigg),
	\end{aligned}
	\end{equation}
	is is immediate to see that $\Y$ satisfies assumption \ref{ass:solution_coef}.
	We hence find that the following rough equation holds
	\[ 
	\dd \Psi = (\dd \Y^1 \partial _x +\dd \Y^0)\Psi + \dd\Y^{-1} 
	\]
	in the sense of Definition \ref{def:Q} together with \eqref{Q_Y}.
	Hence, one sees that $\Psi$ satisfies a similar equation as $\Phi$, but this time with the zero-order multiplication term $Y^0:=\partial _x Z.$ Iterating Proposition \ref{pro:product} yields this time for $n\geq 2$ and $(s,t)\in \Delta$
	\[ 
	\begin{aligned}
	&(\Psi ^n)_{st} -\Psi ^{n,\natural}_{st}
	= \Big(Y^1_{st}\partial_x + n Y^0_{st}\Big)\Psi ^n + nY^{-1}_{st}\Psi^{n-1} 
	\\
	&\quad 
	+ \left(\frac12 (Y^1_{st})^2\partial_{xx} + \left [\bb{Y^1}{\partial_x Y^1}_{st}+ nY^1_{st}Y^0_{st}\right ]\partial_x +n\bb{Y^1}{\partial_x Y^0}_{st}   + \frac{n^2}{2}(Y^0_{st})^2\right) \Psi ^n_s
	\\
	&\quad + \Big(nY^1_{st}Y^{-1}_{st}\partial_x + n(n-1)Y^0_{st}Y^{-1}_{st} +n[\bb{Y^1}{\partial_x Y^{-1}}+\bb{Y^0}{Y^{-1}}]_{st}\Big)\Psi _s^{n-1} 
	\\
	&\quad \quad \quad \quad 
	+ \frac{n(n-1)}{2}(Y^{-1}_{st})^2\Psi _s^{n-2}\enskip.
	\end{aligned}
	\]
	We now aim to estimate the $\mathcal D^\alpha_{Q\{Y\}}$-norm of $\Psi^n$, for which it suffices to estimate the remainder $(\Psi ^{ n})^{\natural}$ (thanks to Proposition \ref{pro:apriori}). 
	The proof of the corresponding estimate hardly differs from that of \ref{transport_i} hence we content ourselves with sketching it. As for evaluating $\Phi^{\natural,n}$, we apply $\delta$ on both sides and find a relation similar to \eqref{delta_Phi}. By induction, we find an estimate on $\|\Phi^n\|_{\mathcal D_{Q\{Z\}}}$, and this permits us to test the equation on $\Psi^n$ against the constant $1.$ We obtain a relation of the form \eqref{pre_gron}, and by application of the rough Gronwall Lemma, the estimate
	\[ 
	\|\Psi\|_{L^\infty(L^n)}
	\lesssim C(n)\rho_\alpha(\Y)
	\lesssim \widetilde C_2\enskip.
	\]
	for a constant $\widetilde C_2=\widetilde C_2\left (n,\rho_\alpha(\Z),\rho_\alpha(\partial_x\Z)\right ) >0.$
	Combining with \ref{transport_i}, we infer the estimate
	\begin{equation}
	\label{cheap_sobolev}
	\|\Phi \|_{L^\infty(0,T; W^{1,n})}\leq \widetilde C_1+ \widetilde C_2\,.
	\end{equation} 
	On the other hand, the above computations show a bound on $\|\Psi \|_{\mathcal D^\alpha_{Q\{Z\}}}$, which a fortiori implies the following $C^\alpha(W^{-1,2})$-estimate
	\begin{equation}
	\label{est:D_alpha}
	\|\Phi \|_{C^\alpha (0,T; W^{-1,2})}\leq C(\rho _\alpha (\Z))\,.
	\end{equation}
	To show that $\Phi$ is essentially bounded, we let $n=2$ in the previous computations. The one-dimensional Sobolev inequality,  \eqref{cheap_sobolev} and a classical interpolation argument with \eqref{est:D_alpha} imply that 
	\begin{equation}
	\label{uniform_est_Phi}
	\|\Phi \|_{L^\infty([0,T]\times \T^1)}\leq C\left(n,d,\alpha ,\rho_\alpha (\Z),\rho _\alpha (\partial_x \Z)\right)\,,
	\end{equation} 
	with a constant depending on the indicated quantities, but not on $\Phi .$

	We can now conclude: by the assumptions on $\Z$, we can consider a sequence $\Z(m)\to _{\rho_\alpha }\Z$ such that $\rho _\alpha (\partial_x \Z(m))\leq C$ uniformly in $m\in\mathbb N.$ Therefore, we obtain the uniform estimate \eqref{uniform_est_Phi}. By a compactness argument similar to \cite[Section 4]{hocquet2018ito}, we can extract a subsequence such that the corresponding $\Phi (m_k)$ converges to a solution of \eqref{transport_gene}. Since the bound \eqref{uniform_est_Phi} is uniform, the limit $\Phi $ satisfies a similar estimate. This finishes the proof of \ref{transport_iii}.
	
	The property \ref{transport_ii} follows by an approximation argument, using similar stability results as in \ref{transport_i} (with a multiplictive rough input instead of an additive one) and the same compactness argument as in \ref{transport_iii}. We leave the details to the reader.
\end{proof}

\begin{remark}
	Stirling's formula $n!\sim \sqrt{2\pi n }(\frac{n}{e})^{n}$ implies that the constant $C(n)$ in \ref{transport_i} is estimated above by a constant times $n$.
\end{remark}

\begin{remark}
	At a formal level the solution $\Phi $ of \eqref{transport_gene} is given by the rough Duhamel formula
	\begin{equation}
	\label{mild_Phi}
	\Phi _t= \int _0^t \exp (\Z\cdot \nabla )_{rt}[\dd (\Z^{-1}_r)](x)\,,
	\end{equation}
	where $\exp (\Z\cdot \nabla )_{rt}$ denotes the flow map associated with the homogeneous rough transport equation.
	In \cite{catellier2016averaging}, it is shown that after time-integration, the latter has a regularizing effect with respect to space, at least if one assumes that (we let $d=1$)
	\begin{itemize}
		\item
		$Z^1_{t}$ is independent of the spatial variable and \textit{$(\rho ,\gamma )$-irregular}. 
		\item $Z^{-1}_t(x):= \sigma(x)\gamma_t$,  where $\gamma$ is also independent of space and $C^1(0,T)$, while $\sigma$ is only a distribution.
	\end{itemize}
	Assuming that $Z$ satisfies these conditions, it should be possible to take advantage of the time-roughness of the transport noise in \eqref{transport_gene}, e.g.\ to show similar results for coefficients with lower spatial regularity. We leave this problem for future investigations.
\end{remark}

\section{Proof of the second uniqueness result, Theorem \ref{thm:regular}}
\label{sec:uniqueness_2}

Let $d=1$ and $\X=(X,\LL)$ be as in the assumptions.
Introduce the following backward dual equation with unknown $m_t(x)$
\begin{equation}
\label{dual_m}
\left \{\begin{aligned}
&\dd m + (\partial _x(a\partial _xm) - b\partial_x m)\dd t= -\dd \B^*m\,,\quad \text{on}\enskip [0,T]\times \T^1\,,
\\
&m_T(x)=1\,,
\end{aligned}\right .
\end{equation}
where $\lambda \leq a(t,x)\leq \lambda^{-1}$ denotes an arbitrary family of coefficients as in Assumption \ref{ass:A}, $b$ belongs to $L^{2r}(L^{2q})$ with $(r,q)$ as in \eqref{conditions_strict}, and $\B=\B\{\X\}$ is the unbounded rough driver of Example \ref{example:URD}. The above equation is understood in the sense of Definition \ref{def:backward_solution}, with $\mathbf P=(P^1,P^2)$ as in Example \ref{exa:B_star}. Existence and uniqueness are guaranteed by Corollary \ref{cor:backward}.

\subsection{Boundedness below}

First, by Proposition \ref{pro:equivalence}, observe that $m$ is a solution of \eqref{dual_m} if and only if $\check m_t:=m_{T-t}$ solves the forward problem
\begin{align}
\label{dual}
&\dd \check m_t +\left(-\partial _x(\check a\partial _x\check m_t) + \check b\partial _x\check m\right)\dd t 
= (\dd\Y^1_t\partial _x  + \dd\Y^0_t) \check m_t \,,\quad \text{on}\enskip [0,T)\times \T^1\,,
\\
\nonumber
&\check m_0=1\enskip ,
\intertext{where $\check a(t,\cdot)=a(T-t,\cdot)$, $\check b(t,\cdot)=b(T-t,\cdot)$, and with the coefficients}
&\Y:=\left( \left [0,-\partial_x X,-X\right ]; \left [\partial_x\LL- \frac12(\partial_x X)^2, \LL\right ];0\right ).
\end{align}
This is indeed a consequence of the following identity
\[
\begin{aligned}
\bb{X}{\partial_{xx}X}
&=\partial _x\bb{X}{\partial_x X} - \bb{\partial _xX}{\partial _xX}
\\
&=\partial _x\LL - \frac12(\partial _xX)^2,
\end{aligned}
\]
valid when $t\mapsto X_t$ has finite variation, and in general by geometricity of $X$ (considering approximating sequences).

We are going to prove the following.
\begin{proposition}
	\label{clm:below}
	There exists $T_+ \in(0,T]$ such that $m$ remains positive up to $t=T_+$, namely
	\begin{equation}
	\label{m_bounded_below}
	\inf_{[0,T_+]\times \T^1} \check m>0\,.
	\end{equation} 
\end{proposition}

\begin{proof}
	Let $\Phi $ be an $L^\infty$-solution of the transport equation 
	\[\left\{\begin{aligned}
	&\dd \Phi =-\dd \X \partial _x \Phi -\dd (\partial_x\X),\quad \text{on}\quad[0,T]\times \T^1,
	\\
	&\Phi_0=0,
	\end{aligned}\right.
	\]
	understood as \eqref{transport_gene} with $\Z$ being defined as
	\[ \Z:=\left( \left [-\partial_x X,0,-X\right ]; \left [0, \LL\right ];\partial_x\LL- \frac12(\partial_x X)^2\right ).
	\]
	and whose existence is guaranteed by Theorem \ref{thm:transport} (notice that $\Y$ and $\Z$ are indeed different since the input $-\partial_x X$ inside $\Z$ is of additive type, while it appears in \eqref{dual} as a multiplicative term). Since the last $d$- coefficients of $Y$ and $Z$ agree, we can use Proposition \ref{pro:product}.
	Introducing the new unknown, 
	\[
	z _t:= \exp(-\Phi _t)\check m-1\,,
	\]
	a simple calculation using \eqref{concl:prod} shows that $z$ is a solution of the following equation with transport rough input
	\begin{equation}
	\label{dual_sigma}
	\left\{
	\begin{aligned}
	&\dd z  +(\mathscr L z -\partial _xf^1 +f^0) \dd t 
	=-\dd \X _t\partial _xz _t\quad \text{on}\enskip [0,T]\times\T^1\,,
	\\
	&z _0(\cdot )=0\,,
	\end{aligned}\right.
	\end{equation}
	where
	\begin{equation}
	\mathscr L z:= \partial _x(\check a\partial _xz) + \partial _x(\check a\partial _x\Phi z ) +(\partial _x\Phi \check a-\check b)\partial _x z+ \check a(\partial _x\Phi)^2 z 
	\end{equation} 
	and
	\[\begin{aligned}
	&f^1:=\check a\partial _x\Phi 
	\\
	&f^0:=\check a(\partial _x\Phi)^2 -\check b\partial _x\Phi \,.
	\end{aligned}
	\]
	This corresponds to the ansatz  \eqref{parabolic_gene}.
	Indeed, we have for instance
	\[ \|f^1\|_{2r,2q}\leq\lambda^{-1}\|\partial_x \Phi\|_{2r,2q} \]
	and thanks to H\"older Inequality
	\[ \|f^0\|_{r,q}\leq \lambda^{-1}\|\partial_x\Phi\|_{2r,2q}^2 + \|b\|_{2r,2q}\|\partial_x\Phi\|_{2r,2q} \enskip.\]
	But the latter terms are estimated by Theorem \ref{thm:transport} which states that for arbitrary $p\in[1,\infty),$ there is a finite constant $C(p)$ such that
	$\|\Phi\|_{L^\infty(L^p)}+\|\partial_x\Phi\|_{L^\infty(L^p)} \leq C(p).$ Choosing $p\geq 2q$ therefore shows that $f^1,f^0$
	satisfy the hypotheses of Theorem \ref{thm:parabolic_gene}.
	Arguing similarly for the coefficients $F^1:=\check a\partial_x\Phi$, $b^1:=\partial_x \Phi\check a-\check b$ and $c:= \check a(\partial _x\Phi)^2$, we obtain the desired conclusion.

	We can now apply Theorem \ref{thm:boundedness}, jointly with Theorem \ref{thm:transport} (ensuring that $\|\Phi\|_{L^\infty([0,T]\times \T^1)}<\infty$): since by definition $\check m_t=\exp(\Phi _t)(1+z_t)$, we infer by the estimate \eqref{L_infty_rho} and H\"older inequality that
	\begin{multline*}
	\inf _{[0,T]}\check m 
	\geq \essinf _{[0,T]\times\T^1}\{\exp\Phi\} -  \esssup _{[0,T]\times\T^1}\big\{z\exp \Phi \big\}
	\\
	\geq \exp(-\|\Phi \|_{L^\infty}) 
	- C_\lambda \Big(\|\partial _x\Phi\|_{1,1}^{1/2} + \|\partial _x\Phi \|_{2,2}
	+ \|\check b\|_{2r,2q} + \|\partial _x\Phi\|_{\frac{2r}{2r-1},\frac{2q}{2q-1}}
	\\
	+\|z\|_{\frac{2r}{r-1},\frac{2q}{q-1}}\Big)\exp(\|\Phi \|_{L^\infty})\enskip.
	\end{multline*}
	Using H\"older and the interpolation inequality \eqref{interp_u}, we see that provided $\epsilon =\epsilon (r,q)>0$ is chosen as in the relation \eqref{choice_eps},
	then 
	\[
	\|z\|_{\frac{2r}{r-1},\frac{2q}{q-1}} \leq \|z\|_{\frac{2r}{r-1}(1+\epsilon ),\frac{2q}{q-1}(1+\epsilon )}\|1\|_{\frac{2r}{r-1}(1+\frac1\epsilon ),\frac{2q}{q-1}(1+\frac1\epsilon )}
	\leq \widetilde C T^{\delta (\epsilon )}
	\]
	where 
	\[
	\delta (\epsilon ):=
	\begin{cases}
	\frac{r-1}{2r(1+\frac1\epsilon )}\enskip \text{if}\enskip r\in(1,\infty)
	\\
	\frac{1}{1+\frac1\epsilon }\enskip \text{if}\enskip r=\infty\,.
	\end{cases}
	\]
	Proceeding similarly for the other terms, we see that choosing $0<T_+\leq T$ small enough (depending only on $r,q,$ $\|b\|_{2r,2q}$, $\lambda$ and $\rho_\alpha(\X),\rho_\alpha(\partial_x\X)$) guarantees that $\inf _{[0,T_+]\inf\T^d}\check m>0.$ The conclusion follows.
\end{proof}

\subsection{Proof of Theorem \ref{thm:regular}}
We now make use of the same notations as that of Section \ref{sec:uniqueness}.
For $n\in\mathbb N,$ going back to the relation \eqref{pre_gron}, we have for any $0\leq s\leq t\leq T$ and $\phi \in W^{3,\infty}$:
\begin{multline}
\Big(\int_{\mathbb T^1}\beta _n(v)\dd x\phi \Big)_{st}
+\iint_{[s,t]\times\T^1}\Big [\beta _n''(v)a(u^1)(\partial _xv)^2\phi
+a(u^1)\partial _j(\beta _n(v))\partial _x\phi
\\
+\beta _n''(v)\partial _xv\varDelta\partial _xu^2\phi + \beta '_n(v) \varDelta\partial _xu^2\partial _x\phi 
\Big]\dd x\dd r
\\
=\int_{\T^1}\beta _n(v)(B^{1,*}_{st} + B^{2,*}_{st})\phi \dd x + \langle \beta _n(v)^\natural,\phi \rangle
\end{multline}
where $\varDelta:=a(u^1)-a(u^2).$
Equivalently, one may write
\begin{multline}
\Big(\int_{\mathbb T^1}\beta _n(v)\dd x\phi \Big)_{st}
+\iint_{[s,t]\times\T^1}\beta _n''(v)\Big[a(u^1)(\partial _xv)^2 + \varDelta \partial _xv \partial _xu^2\Big]\phi \dd r\dd x
\\
=\int_s^t\Big \langle  [A - \partial _x(b_n\cdot )]\beta _n(v) + \partial _xg_n,\phi \Big\rangle
\\
+\int_{\T^1}\beta _n(v)(B^{1,*}_{st} + B^{2,*}_{st})\phi \dd x + \langle \beta _n(v)^\natural,\phi \rangle
\end{multline}
where, denoting by $a_n$ the decreasing sequence defined in \eqref{sequence_an},
we introduce the velocity and flux terms
\[\left[\begin{aligned}
&b_n(v):= \mathbf 1_{v>a_n}\frac{\beta _n'(v)\big(a(u^1)-a(u^2)\big)}{\beta _n(v)}\partial _xu^2 
\\
&g_n:=b_n(\beta '_n(v)v-\beta_n(v))\,.
\end{aligned}\right.
\]
It is easily seen that
\begin{equation}
\label{strong_bn}
b_n(v)\to b(v):= \mathbf 1_{v\neq 0}\frac{a(u^1)-a(u^2)}{v}\partial _xu^2\enskip \text{in}\enskip L^{2r}(L^{2q})\text{-strong.}
\end{equation} 
Moreover, using Taylor Theorem, \eqref{sequence_an} and \eqref{derivees} give the estimate:
\[
\mathbf 1_{v>a_n}|\beta _n'(v)v-\beta (v)|\leq C|v|
\int_{a_n/v}^{a_{n-1}/v}|v\beta _n''(\tau v)|\dd \tau 
\leq \widetilde C\,,\quad \dd t\otimes \dd x\text{-almost everywhere,}
\]
and thus the flux is estimated thanks to H\"older Inequality and \eqref{interpolation_inequality} as
\begin{equation*}
\|g_n\|_{L^2(L^2)}\leq \widetilde C\|b_n\|_{2r,2q}\|v\|_{\frac{2r}{r-1},\frac{2q}{q-1}}\leq \hat C_{r,q}\|b_n\|_{2r,2q}(\|v\|_{\infty,2}+\|\nabla v\|_{2,2})
\,.
\end{equation*} 
Since on the other hand, $g_n\to 0$ $\dd t\otimes \dd x$ almost everywhere, we see by dominated convergence that 
\begin{equation}
\label{strong_Fn}
g_n\to 0\quad \text{in}\enskip L^2(L^2)\text{-strong.}
\end{equation}

Next, for each $n\in\N,$ let $m^n$ be the solution of the dual backward problem:
\[\left\{
\begin{aligned}
&\dd m^n + (A + b_n\partial _x)m^n\dd t = (\dd \X\partial _x -\partial_x\X)m^n\quad \text{on}\enskip [0,T]\times\T^1\,,
\\
&m^n_T=1\,.
\end{aligned}\right.
\]
By the product formula, Proposition \ref{pro:product_backward}, we observe that
\begin{multline}
\Big(\int_{\mathbb T^1}\beta _n(v)m^n\dd x\phi \Big)_{st}
+\iint_{[s,t]\times\T^1}\beta _n''(v)a(u^1)(\partial _xv)^2m^n\dd x\dd r
\\
=\iint _{[s,t]\times\T^1}\Big[\beta _n''(v)\partial _xv\varDelta\partial _xu^2m^n + g_n\partial _xm^n \Big]\dd x\dd r
\end{multline}
for each $(s,t)\in\Delta .$
By \eqref{strong_bn}, \eqref{strong_Fn} and the continuity part of Theorem \ref{thm:parabolic_gene}, we obtain at the limit
\[\Big(\int_{\mathbb T^1}|v|m\dd x\phi \Big)_{st}\leq 0\,.
\]
where $m$ is the solution of the dual equation associated with $b.$ Since $b\in L^{2r}(L^{2q}),$ $m$ is bounded below by Proposition \ref{clm:below}. This shows that $v=0$ and finishes the proof of Theorem \ref{thm:regular}.\hfill\qed

\section{Higher regularity: proof of Theorem \ref{thm:existence_2}}
\label{sec:higher}
Our aim here is to prove the second existence theorem, Theorem \ref{thm:existence_2} for the one-dimensional quasilinear equation
\begin{equation}
\label{equation_quasilin_last}
\left\{
\begin{aligned}
&\dd u -\partial _x(a(t,x,u)\partial _xu)\dd t=\dd \X\partial _xu\,,\quad \text{on}\enskip (0,T]\times \T^1 \,,
\\
&u_0=u^0\in W^{1,2}(\T^1)\,.
\end{aligned}\right.
\end{equation}

Note that by definition of a solution, one has the following Euler-Taylor expansion in $W^{-2,2}$:
\begin{equation}
\label{equation_un_dev}
u_{st} - \int _s^t\partial _x(a(u)\partial _xu)\dd r
= \left(X_{st}\partial _x + \frac12(X_{st})^2\partial _{xx} + \LL_{st}\partial _x\right)u_s + u^\natural_{st}
\end{equation} 
where the remainder $u^{\natural}$ belongs to $\CC(0,T;W^{-3,2}).$

\paragraph*{The $L^\infty(0,T;W^{1,2})\cap L^2(0,T;W^{2,2})$ a priori estimate.}
For simplicity, in the sequel we denote by $a_x(u)= D_2a(t,x,u(t,x))$ and $a_{z}(u)=D_3a(t,x,u(t,x))$.
We follow the pattern of the proof of \ref{transport_iii} in Theorem \ref{thm:transport}:
taking spatial derivatives in \eqref{equation_un_dev}, we see that the new unknown $v:=\partial _xu$ satisfies
\begin{equation}
\label{eq:derivee_ex}
\begin{aligned}
\int_{\T^1}v_{st}\phi \dd x 
&+ \iint_{[s,t]\times\T^1}a(u)\partial _xv \partial _x\phi \dd x\dd r
\\
&
=\iint_{[s,t]\times\T^1}\partial _x(a_x(u)v+a_z(u)v^2)\phi \dd x\dd r
+\left\langle\left(X_{st}\partial _x + \partial _xX_{st}\right)v_s,\phi \right\rangle
\\
&
\quad \quad 
+\left\langle\left(\frac12X^2_{st}\partial _{xx} + \Big(\partial _xX_{st}X_{st}+\LL_{st}\Big)\partial _x + \partial _x\LL_{st}\right)v_s,\phi \right\rangle + \langle v^{\natural}_{st},\phi \rangle\,,
\\
&=:\int_{[s,t]\times \T^1}\partial _x(a_x(u)v+a_z(u)v^2)\phi \dd r\dd x
+\left\langle\left(B^{(1),1}_{st} + B^{(1),2}_{st}\right)v_s,\phi \right\rangle
+\langle v^{\natural}_{st},\phi \rangle
\end{aligned}
\end{equation} 
where for $n\in\N$ we denote by $\B^{(n)}$ the unbounded rough driver given by
\[\left[\begin{aligned}
B^{(n),1}_{st}
&:=X_{st}\partial _x  + n\partial _xX_{st}
\\
&=: Y_{st}^1\partial _x + nY^0_{st}
\\
B^{(n),2}_{st}
&:=\frac12(X_{st})^2\partial _{xx}+ (\LL _{st}\partial _x +n\partial _xX_{st}X_{st})\partial _x +n\partial _x\LL _{st}
\\
&=:\frac12(Y^1_{st})^2\partial _{xx} + \big(\bb{Y^1}{\partial _x Y^1 }_{st}+ nY^0_{st}Y^1_{st}\big)\partial _x + n\bb{Y^1}{\partial _xY^0}_{st} + n^2\frac12(Y^0_{st})^2
\end{aligned}\right.
\]
and where we let $\bb{Y^1}{\partial _xY^0}_{st}:=\LL_{st}$, while
$\bb{Y}{\partial _xY^0}_{st}:=\partial _x\LL_{st}-\frac12(\partial _xX_{st})^2.$
Taking the product of $v$ with itself, we see by the product formula given in \cite[Proposition 4.2]{hocquet2018ito} (this version is also contained in Appendix \ref{app:product}) that
\begin{equation*}
\begin{aligned}
\int_{\T^1}v^{2}_{st}\phi \dd x 
&+ 2\iint_{[s,t]\times \T^1}a(u)\big[(\partial _xv)^2 \phi + v\partial _xv \partial _x\phi \big]\dd x\dd r
\\
&
=-2\int_s^t\Big\langle N(v),\phi \Big\rangle\dd r
+\left\langle\left(B^{(2),1}_{st} + B^{(2),2}_{st}\right)v_s^n,\phi \right\rangle
+\langle v^{2,\natural}_{st},\phi \rangle
\end{aligned}
\end{equation*}
where we denote by $N(v)$ the non-linearity
\[
N(v):=a_x(u)v^2 + a_z(u)v^2\partial_x v -\partial_x(a_z(u)v^3) 
\]
Using Young Inequality, and estimating the $L^3$ and the $L^2$ norms by that of $L^4$, we have
for $|\phi |_{W^{1,\infty}}$
\[\begin{aligned}
\langle N(v),\phi \rangle
&:=\int_{\T^1}\big[a_x(u)v^2+a_z(u)\big(v^{2}\partial _xv\phi +v^{3}\partial _x\phi) \big]\dd x
\\
&\leq C_\lambda \big(\esssup_{t\in[0,T]}|a(t,\cdot,\cdot)|_{C^1_b}\big)(1+|v|^4_{L^4}) + \frac{1}{\lambda }|\partial _xv|^2\,.
\end{aligned}
\]

Thus, taking $\phi =1,$ and applying Proposition \ref{pro:apriori}, we get for $|t-s|\leq L$ small enough
\[\begin{aligned}
&\int_{\T^1}v^{2}_{st}\phi \dd x 
+ \lambda \iint_{[s,t]\times \T^1}\partial _xv^2 \dd x\dd r
\leq 
C\int_s^t[1+|v|^4_{L^4}]\dd r
\\
&\quad 
+\rho _\alpha (\partial_x \X)(t-s)^\alpha \Big[1+\lambda ^{-1}\iint_{[s,t]\times\T^1}\Big((\partial _xv)^2 + |v||\partial _x v| + 1+|v|_{L^4}^4\Big)\dd x\dd r\Big].
\end{aligned}
\]
Taking $L(\lambda ,\rho _\alpha (\X),\rho _\alpha (\partial_x X))>0$ smaller if necessary and then applying Lemma \ref{lem:gronwall} with $\varphi (s,t):=C\int_s^t(1+|v|^4_{L^4})\dd r$
we end up with the following inequality for $t>0$ small enough:
\[
E_t:=\sup_{s\in[0,t]}\int_{\T^1}v^{2}_s\phi \dd x 
+ \iint_{[0,t]\times \T^1}(\partial _xv)^2 \dd x\dd r
\leq C\left(t,\lambda ,\alpha ,L\right)\left (1+ |\partial _xu^0|^2_{L^2}+\int_0^t|v|_{L^4}^4\right )\,.
\]
and where the above constant is non-decreasing with respect to $t>0.$
But the one-dimensional Gagliardo-Nirenberg inequality asserts that
\[
|v|_{L^4}^4\lesssim |v|^3|\partial _xv| + |v|^4_{L^2}\leq C|v|^6_{L^2} + \frac{1}{2C(t,\lambda ,\alpha ,L)}|\partial _xv|^2
\]
and so, absorbing further to the left we obtain the following nonlinear-type relation:
\begin{equation}
\label{nonlinear_relation}
E_t\leq \widetilde C\left (1+\int_0^tE_t^3\right )\,,
\end{equation} 
for another such constant $\widetilde C>0.$
Integrating \eqref{nonlinear_relation}, we see that there exists a positive time $T_*\in(0,T]$ and  $\hat C>0$, both depending only on the quantities $\alpha ,\rho _\alpha (\X),\rho(\partial_x X),$ and $|u^0|_{W^{1,2}}$ such that
\begin{equation}
\label{energy_regular}
\|u\|_{L^\infty(0,T_*;W^{1,2})} + \|u\|_{L^2(0,T_*;W^{2,2})}\leq C\,.
\end{equation} 

\paragraph*{Conclusion.}
We can now repeat the compactness argument of Section \ref{sec:existence}.
Since $\X$ is geometric, it is easily seen by Assumption \ref{ass:geometric} that there is a time-smooth approximating sequence $\X(n)\in\mathscr C^\alpha $ such that $\partial _x\X(n)\in\mathscr C^\alpha $ for each $n\in\N$ while 
\[
\rho_\alpha (\X(n),\X)+ \rho _\alpha (\partial _x\X(n),\partial _x\X)\to 0\,.
\]
By \eqref{energy_regular} and the arguments used in the proof of the first existence theorem, we see that any limit point $u$ of the corresponding sequence of solutions $u(n)$ ought to solve the equation \eqref{equation_quasilin_last}.
At the limit, we see that \eqref{energy_regular} still holds for the same $T_*>0$, which means in particular that $\partial _xu$ belongs to $L^\infty(0,T_*;W^{1,2}).$
By the Sobolev embedding $H^1\hookrightarrow L^4$, the condition \eqref{integ:intro} is fulfilled for $\partial _xu$ and the exponents $(r,q)=(\infty,2)$, meaning that the above solution is unique, locally in time.
Now, repeating the argument shows the existence and uniqueness of the maximal solution, and this finishes the proof of the Theorem.
\hfill\qed

\begin{appendix}
	\setcounter{equation}{0}
	\renewcommand{\theequation}{\thesection.\arabic{equation}}
\section{Sewing Lemma}
\label{app:sewing}

For the reader's convenience, we recall the statement of the Sewing Lemma in a Banach space $(E,|\cdot|)$, as formulated for instance in \cite{gubinelli2010rough,deya2016priori}.

\begin{theorem}[Sewing Lemma]
	\label{thm:sewing}
	Let $H \colon \Delta \rightarrow E$ and $C>0$ be such that 
	\begin{equation}
	\label{a_gamma}
	\left|\delta H_{s\theta t}\right|\leq C\omega (s,t)^{a }\,
	, \quad 0 \leq s \leq \theta \leq t \leq T
	\end{equation}
	for some $a > 1$, and some control function $\omega ,$ and denote by $[\delta H]_{a,\omega }$ the smallest possible constant $C$ in the previous bound. 
	
	There exists a unique pair $I\colon [0,T] \rightarrow E$ and $I^{\natural} : \Delta \rightarrow E$ satisfying
	\[
	I_{t}-I_s = H_{st} + I_{st}^{\natural}
	\]
	where for $0\leq s\leq t\leq T$,
	\[
	|I_{st}^{\natural}| \leq C_a [\delta H]_{a,\omega } \omega (s,t)^{a}\,,
	\]
	for some constant $C_a$ only depending on $a$. In fact, $I$ is defined via the Riemann type integral approximation
	\begin{equation} \label{RiemannSum}
	I_t = \lim \sum_{ i=1 }^n H_{t^n_i t^n_{i+1}} \,,
	\end{equation}
	the above limit being taken along any sequence of partitions $\{t^n,n\geq 0\}$ of $[0,t]$ whose mesh-size converges to $0$.
\end{theorem}

\section{A product formula}
\label{app:product}

In this section, we show a product formula for two solutions of a similar problem.
Fix $p,p'\in [1,\infty]$ with $1/p+1/p'=1,$ and consider two controlled paths $u,v$
where
\begin{equation}
\label{paths_uv}
\begin{aligned}
&u\in L^\infty(0,T;L^p)\cap L^p(0,T;W^{1,p})\,,
\\
&v\in L^\infty(0,T;L^{p'})\cap L^{p'}(0,T;W^{1,p'})\,,
\end{aligned}
\end{equation}
and we assume furthermore that on $[0,T]\times \T^d$:
\begin{align}
\label{equation_u}
&\dd u= F\dd t + (\dd \X^i\partial _i + \dd \Y^0) u + \dd \Y^{-1},
\quad \text{in}\enskip L^p\,,
\intertext{while}
\label{equation_v}
&\dd v=G\dd t +(\dd \X^i\partial _i + \dd \Z^0) u + \dd \Z^{-1},
\quad \text{in}\enskip L^{p'}\,
\end{align}
for some $F\in L^1(0,T;W^{-1,p})$ and $G\in L^1(0,T;W^{-1,p'}),$ as energy solutions. The above is to be understood as the sytem of rough PDEs
\[\left\{
\begin{aligned}
&\dd u= F\dd t + \dd \Q\{\Y\}(u),
\\
&\dd v=G\dd t + \dd \Q\{\Z\}(v)
\end{aligned}\right.
\]
where $\Q\{\Y\}, \Q\{\Z\}$ are as in the formula \eqref{Q_Y}, with respective triads
(by convention $Y^i=X^i=Z^i$ for $i=1,\dots d$)
\[\begin{aligned}
&\Y=\Big([X^i]_{i=-1,0,\dots ,d}\,;\,[\bb{X^j}{\partial _j Y^i}]_{i=0,\dots ,d}\,;\,\bb{X^j}{\partial _j Y^{-1}}+\bb{Y^0}{Y^{-1}}\Big)
\\
&\Z =\Big([X^i]_{i=-1,0,\dots ,d}\,;\,[\bb{X^j}{\partial _j Z^i}]_{i=0,\dots ,d}\,;\,\bb{X^j}{\partial _j Z^{-1}}+\bb{Z^0}{Z^{-1}}\Big).
\end{aligned}
\]

The reason for restricting to this situation is twofold.
First, it is good enough since in all the manuscript, we are only interested in the case where $v=\beta (u)$ where $\beta $ is a Nemytskii operator (with the exception of the forward-backward formula proven in Section \ref{subsec:backward}). Indeed, in that case the transport part of the equation on $v$ will always be the vector field $\XX=X^i\partial _i$.
Second, as is quickly realized, is is generally not possible to obtain a product formula for \textit{any} pair of coefficients $\Y$ and $\Z$, at least without further knowledge such as the existence of a \textit{joint lift} for $(Y,Z)$ (providing values for crossed integrals). Even in the case when such joint lift is available, it is not clear how to write down a rough PDE of the form \eqref{rough_PDE_gene} for the pointwise product $uv$, and hence we prefer to leave this question for further investigations.

Prior to give the content of the product formula, it is convenient to introduce some notation.
\begin{notation}
	\label{nota:composition}
	As in Example \ref{example:bracket},
	denote by $V\mapsto \widehat V=V^i\partial_i$ the natural isomorphism between coefficients and vector fields (i.e.\ derivations). 
	If $A=(A^i_t(x))_{i=1,\dots d},\, B=(B_{t}(x))$ are such that the symbol $\bb{A}{B}$ is known, we adopt the following notation for conciseness	
	\[
	\bb{\widehat A}{B}_{st}:=\bb{A^i}{\partial_i B}_{st},
	\quad\forall(s,t)\in\Delta.
	\]
	Furthermore assuming the knowledge of $\bb{A}{B}$ for $i=1\dots d$, we introduce another bilinear operation ``$\odot$'', which combines integration with composition in the following way:
	\[
	\begin{aligned}
	&\bbb{\widehat A}{B}_{st}
	:=\bb{\widehat A}{B}_{st}
	+\bb{A^i}{B}_{st}\partial _i
	\\
	&\bbb{B}{\widehat A}_{st}
	:=\bb{B}{A^i}_{st}\partial _i \enskip.
	\end{aligned}
	\]
	
	Finally, if $\widehat V$ is another vector field, and if the symbols $\bb{A^i}{V^j}, \bb{A^i}{\partial _iV^j}$ are known, we define consistently:
	\[
	\bbb{\widehat A}{\widehat V}_{st}
	:=\bb{\widehat A}{ V^j}_{st}\partial_j
	+\bb{A^i}{B^j}_{st}\partial _{ij}\,.
	\]
	Namely, when $A,V$ have finite variation, the latter corresponds to the integral $\int_s^t \widehat{\dot A_r}\circ \widehat{V}_{sr}$.
\end{notation}

The main result of this appendix is the next proposition.
\begin{proposition}[Product formula]
	\label{pro:product}
	Let $u,v$ be as in \eqref{paths_uv}--\eqref{equation_v}, and assume that $\Y,\Z$ are geometric and compatible for product.
	\[
	\begin{aligned}
	F=\partial _i  f^i +f^0\,,\quad f \in L^1(0,T;(L^p)^{d+1})\,,
	\\
	G=\partial _ig^i +g^0\,,\quad g \in L^1(0,T;(L^{p'})^{d+1})
	\end{aligned}
	\]
	Assume that for $i =0,\dots ,d,$ 
	\[\sup_{a\in \R^d\text{ with }|a|\leq 1}
	\left(\|\partial _i  u(\cdot )g^i(\cdot -a)\|_{L^1(0,T;L^1)}+
	\|f^i(\cdot -a) \partial _i  v(\cdot )\|_{L^1(0,T;L^1)}\right)<\infty\,.
	\]
	
	The following assertions hold.
	\begin{enumerate}[label=(\arabic*)]
		\item \label{app_Q_shift}
		The two-parameter mapping $\Q\equiv(Q_{st}^{1},Q_{st}^{2})_{(s,t)\in\Delta}$ defined for column vectors $U\in W^{2,p}(\T^d;(\R^3)^T)$ via
		\begin{equation*}
		\setlength{\arraycolsep}{1pt}
		Q^1(U) 
		:=
		\begin{pmatrix}
		\XX +Y^0&
		0 &
		0\\
		0 &
		\XX +Z^0&
		0\\
		Z^{-1} &
		Y^{-1} &
		\XX +Y^0 + Z^0
		\end{pmatrix}
		U + 
		\begin{pmatrix}
		Y^{-1}\\
		Z^{-1}\\
		0
		\end{pmatrix}
		\end{equation*}
		and 
		\begin{align*}
		\setlength{\arraycolsep}{1pt}
		&Q^2(U) 
		:= 
		\begin{pmatrix}
		\displaystyle (\XX+Y^0)^{\odot2} 
		&
		0 &
		0
		\\[0.9em]
		0 &
		\displaystyle (\XX+Z^0)^{\odot2}
		&
		0\\[0.7em]
		\substack{\displaystyle +\bb{(\XX+Z^0)}{Z^{-1}} 
			\\
			\displaystyle +Z^{-1}\XX+Z^{-1}Y^0} 
		&
		\substack{\displaystyle +\bb{(\XX+Y^0)}{Y^{-1}} 
			\\
			\displaystyle +Y^{-1}\XX+Y^{-1}Z^0} 
		&
		\displaystyle (\XX + Y^0+Z^0)^{\odot2} \\
		\end{pmatrix}
		U
		\\[0.7em]
		&\quad \quad \quad \quad \quad 
		+ 
		\begin{pmatrix}
		\bb{\XX}{Y^{-1}} + \bb{Y^0}{Y^{-1}}\\
		\bb{\XX}{Z^{-1}} + \bb{Z^0}{Z^{-1}}\\
		Y^{-1}Z^{-1}
		\end{pmatrix}\,,
		\end{align*}
		is an unbounded rough driver. Namely, denoting by $\widetilde Q^1:=Q^1-Q^1(0),$  $\Q$ satisfies the affine Chen's relations
		\[\begin{aligned}
		\delta Q^1_{s\theta t}=0
		\\
		\delta Q^2_{s\theta t}=\widetilde Q^1_{\theta t}\circ Q^1_{s\theta }\,
		\end{aligned}
		\quad \text{for}\enskip (s,\theta,t)\in \Delta_2.
		\]
		
		\item \label{app_prod_uv}
		The pointwise product $uv$ belongs to the controlled path space $\mathcal D_{Q},$
		and moreover the enlarged unknown
		\[
		U:=
		\begin{pmatrix}
		u\\
		v\\
		uv
		\end{pmatrix}
		\]
		satisfies the following system of equations
		\[
		\dd U= \big(F,G,uG +Fv\big)^{T}\dd t + \dd \Q^i(U) \,.
		\]
	\end{enumerate}
\end{proposition}

In particular, Proposition \ref{pro:product} implies the following Euler-Taylor expansion for the third component: for every $(s,t)\in \Delta$
\begin{multline}
\label{concl:prod}
(uv)_{st} - \int_s^t\big[uG +Fv\big]\dd r -(uv)^{\natural}
\\
= \Big(\XX_{st} + Y^0_{st}+Z^0_{st}\Big)(u_sv_s) + Y_{st}^{-1}v_s + Z_{st}^{-1}u_s 
\\
+\Big(Z^{-1}_{st}\XX_{st} 
+\bb{\XX}{Z^{-1}}_{st}
+\bb{Z^0}{Z^{-1}}_{st}
+Y^0_{st}Z^{-1}_{st}
\Big)u_s
\\
+\Big(Y^{-1}_{st}\XX_{st}+
\bb{\XX}{Y^{-1}}_{st}+\bb{Y^0}{Y^{-1}}_{st} +Y^{-1}_{st}Z^0_{st}\Big)v_s
\\
+\Big(\XXX_{st} + (Y^0_{st}+Z^0_{st})\XX_{st} +\bb{\XX}{Y^0}_{st}+\bb{\XX}{Z^0}_{st} +\frac12(Y^0_{st})^2 +Y^0_{st}Z^0_{st}+  \frac12(Z^0_{st})^2\Big)(u_sv_s)
\\
+Y^{-1}_{st}Z^{-1}_{st}
\end{multline}
where $(uv)^{\natural}$ belongs to $\CC(0,T;W^{-3,1}).$ 

\begin{proof}
	This is an immediate generalization of the results of \cite[Section 4]{hocquet2018ito}. Details are left to the reader.
\end{proof}

\begin{remark}
	The equation resulting from the product of two such paths $u,v$ is in general not closed, unless both elements are associated with a zero additive rough input, i.e.\ if $X^{-1}=Y^{-1}=0.$
	This appears as an inconvenient if one aims to use the remainder estimates, Proposition \ref{pro:apriori}, since the statement assumes an ansatz of the form
	\[
	\dd  U = F\dd t + \dd \Q (U).
	\]
	As item \ref{app_prod_uv} shows however, the equation on the vector-valued path $U:=(u,v,uv)$ \textit{does} have a closed form. Furthermore, the matrices $\widetilde Q^1,\widetilde Q^2$ are triangular, with allows to obtain remainder estimates for iterated products by a simple induction (see Section \ref{sec:transport} for an illustration of this idea).
\end{remark}

\end{appendix}

\section*{Acknowledgements}
A.H.\ would like to thank M.\ Hofmanov{\'a} for suggesting the sequence $\beta _n(\cdot )$ used in the proof of uniqueness.
A.H.\ was supported by Deutsche Foschungsgemeinschaft (DFG) via Research Unit (Forschergruppe) FOR 2402 ``Rough paths, stochastic partial differential equations and related topics'', and through grant CRC 910 ``Control of self-organizing nonlinear systems: Theoretical methods and concepts of application,'' Project (163436311).

\end{document}